\documentclass[11pt, twoside]{article}
\pdfoutput=1

\usepackage{graphicx}
\usepackage[caption=false]{subfig}
\usepackage{amsmath}
\usepackage{amssymb,amsfonts}
\usepackage{amsthm}
\usepackage{bm}
\usepackage{mathrsfs}
\usepackage{amssymb}
\usepackage{multirow}
\usepackage[margin=1in]{geometry}

\usepackage{algorithm}
\usepackage{algorithmic}

\numberwithin{equation}{section}
\usepackage{multirow}
\usepackage{makecell}
\usepackage{booktabs}
\theoremstyle{definition}
\newtheorem{theorem}{Theorem}

\newtheorem{lemma}{Lemma}

\newtheorem{remark}{Remark}
\newtheorem{example}{Example}
\newtheorem{assumption}{Assumption}

\usepackage{cite}
\usepackage{hyperref}
\usepackage[nameinlink]{cleveref}

\usepackage{fancyhdr}
\begin{document}
	
	\title{An HDG Method for Time-dependent  Drift-Diffusion Model of Semiconductor Devices}%\tnoteref{mytitlenote}}

\author{Gang Chen%
	\thanks{School of Mathematics Sciences, University of Electronic Science and Technology of China, Chengdu, China (\mbox{cglwdm@uestc.edu.cn}).}
	\and
	Peter Monk%
	\thanks{Department of Mathematics Science, University of Delaware, Newark, DE, USA (\mbox{monk@udel.edu}).}
	\and
	Yangwen Zhang%
	\thanks{Department of Mathematics Science, University of Delaware, Newark, DE, USA (\mbox{ywzhangf@udel.edu}).}
}

\date{\today}

\maketitle

\begin{abstract}
		We propose a hybridizable discontinuous Galerkin (HDG) finite element method to approximate the solution of the  time dependent drift-diffusion problem.  This system involves a nonlinear convection diffusion equation
			for the electron concentration $u$ coupled to a linear  Poisson problem for the the electric potential $\phi$. The non-linearity in this system is the
			product of the $\nabla \phi$ with $u$.  An improper choice of a numerical scheme can reduce the convergence rate. To obtain optimal HDG error estimates for $\phi$, $u$ and their gradients, we utilize two different HDG schemes to discretize the  nonlinear convection diffusion equation and the Poisson equation. We prove optimal order error estimates for the
			semidiscrete problem.  We also present numerical experiments to support our theoretical results.
\end{abstract}

\section{Introduction}\label{intro}
{Drift-diffusion equations play an  important role in modeling the movement of charged particles particularly in semiconductor physics \cite{baltes1986integrated,MR723107,burgler1989new,datta1997electronic,MR796097,jerome2012analysis,MR1437143,MR821965}.  Besides the applications to semiconductors, these kinds of PDEs have many applications in the simulation of batteries \cite{Meng_PNP_Commun_Comput_Phys_2014,wu2002newton}, charged particles in biology \cite{Lu_PNP_JCP_2010,xu2014modeling} and  physical chemistry \cite{eisenberg2010energy,horng2012pnp,tu2013parallel,hsieh2015transport}.}

We consider the following model time dependent drift-diffusion equation posed on a Lipschitz polyhedral domain $\Omega \subset \mathbb R^d (d\ge 2)$: {we seek to determine the unknown  electron {density} $u$ and the electric potential $\phi$ that satisfy}
\begin{subequations}\label{Drift_Diffusion_Main_Equation}
	\begin{align}
	u_t- \Delta u+\nabla\cdot(u\nabla{\phi})&=0&\text{ in }\Omega\times (0,T],\label{Drift_Diffusion_Main_Equation_a}\\
	-\varepsilon\Delta {\phi}+u&=0&\text{ in }\Omega\times (0,T],\label{Drift_Diffusion_Main_Equation_b}\\
	u&=g_u&\text{ on }\partial\Omega\times(0,T],\label{Drift_Diffusion_Main_Equation_c}\\
	{\phi}&=g_{{\phi}}&\text{ on }\partial\Omega\times(0,T],\label{Drift_Diffusion_Main_Equation_d}\\
	u(\cdot,0)&=u_0&\text{ in }\Omega,\label{Drift_Diffusion_Main_Equation_e}
	\end{align}
	{where $\varepsilon$ is a constant and typically small in real applications. In our analysis, we assume $\varepsilon=O(1)$ and have not analyzed the
		$\varepsilon $ dependence of the coefficients. This  will be  considered in  future work.}
	We shall discuss the  smoothness assumptions on $g_u$, $g_{\phi}$ and $u_0$  needed for our analysis later in the paper. {Applications of the drift-diffusion model often involve more 
		complicated versions of the above model, for example including additional particle transport equations (for example, for holes) and recombination terms.  However the above system contains the principle difficulty
		from the point of view of proving convergence: the term $\nabla\cdot(u\nabla \phi)$.}
\end{subequations}

Theoretical and numerical  studies for this type of partial differential equation (PDE) have a long history.  For the theoretical analysis of the drift-diffusion system; see \cite{Dolbeault_PNP_Poincare_2000,Hebisch_Nonlinear_Analysis_1994,Gajewski_ZAMM_1985,Gajewski_JMAA_1986,Mock_SIAM_JMA_1974,Jerome_Book_Semiconductor_Springer_1996} and the  references therein. Computational studies  started in the 1960s \cite{de1968accurate,gummel1964self} and many discretization methods have been used for the drift-diffusion system in the past decades. For an extensive body of literature devoted to this subject we refer to, e.g., the finite difference method \cite{flavell2014conservative,He_PNP_FDM_AMC_2016,Liu_FDM_PNP_JCP_2014,Mirzadeh_PNP_JCP_2014}, the finite volume method \cite{Chainais_FVM_IMA_2003,Chainais_FVM_M3AS_2004,Chainais_FVM_M2AN_2003,Bessemoulin_FVM_2014_SINUM,Bessemoulin_NM_2012}, the standard finite element method (FEM) \cite{Lu_PNP_JCP_2010,Sun_PNP_FEM_JCAM_2016,Gao_FEM_PNP_JSC_2017}, and mixed FEM \cite{He_Mixed_PNP_NMPDE_2017,gao2018linearized}. 
Furthermore, there are many new models in which the drift-diffusion equation coupled with other PDEs; such as Stokes \cite{He_PNP_Stokes_JCAM_2018}, Navier-Stokes \cite{Schmuck_PNP_NS_M3AS_2009} and Darcy flow \cite{Frank_PNP_Darcy_M2AN_2017}.  However these extensions are outside the scope of this paper.

{The product of the gradient of the electric potential, $\nabla \phi$  with electron concentration $u$ 
	in (\ref{Drift_Diffusion_Main_Equation_a})} can cause a reduction in the convergence rate of the solution if the numerical schemes for the two equations are not properly devised. In \cite{Sun_PNP_FEM_JCAM_2016}, they obtained an optimal convergence rate in $H^1$ norm but a suboptimal in $L^2$ norm by the standard FEM. To overcome the convergence order reduction, a new method was proposed to discretize the system \eqref{Drift_Diffusion_Main_Equation};  mixed FEM for Poisson equation \eqref{Drift_Diffusion_Main_Equation_b} and standard FEM for \eqref{Drift_Diffusion_Main_Equation_a},. This scheme provides optimal error estimates for  $u$ and $\phi$ in both {$H^1$ or $H({\rm{}div})$ as appropriate as well as in the $L^2$ norm}. Very recently, the authors in \cite{gao2018linearized} obtained an optimal convergence rate by using  mixed FEM for both \eqref{Drift_Diffusion_Main_Equation_a} and \eqref{Drift_Diffusion_Main_Equation_b}. 

In the drift-diffusion model, typically, the magnitude of $\nabla \phi$ is huge (see \cite{Brezzi_DD_SINUM_1989}). Therefore,  it is natural to consider the discontinuous Galerkin (DG) method to discretize the system \eqref{Drift_Diffusion_Main_Equation}. In \cite{Liu_LDG_DD_SCM_2016}, a local DG (LDG) method was used to study a 1D drift-diffsuion equation, they obtained an optimal convergence rate by using an important relationship between the gradient and interface jump of the numerical solution with the independent numerical solution of the gradient in the LDG methods; see \cite[Lemma 2.4]{Wang_LDG_SINUM_2015} and \cite[Lemma 4.3]{Liu_LDG_DD_SCM_2016}. However, to the best of our knowledge, the inequality in \cite[Lemma 2.4]{Wang_LDG_SINUM_2015} is not straightforward to extend to high dimensions.

Moreover, the number of degrees of freedom for the DG or LDG methods is much larger compared to standard FEM; this is the main drawback of DG methods. Hybridizable discontinuous Galerkin (HDG) methods were originally proposed in \cite{Cockburn_Gopalakrishnan_Lazarov_Unify_SINUM_2009} to remedy this issue.  The global system of HDG methods only involve the degrees of freedom on the boundary face of the element. Therefore, HDG methods have a significantly smaller number of  degrees of freedom in the global system compared to DG methods, LDG methods or mixed FEM.  
Moreover,  HDG methods  keep the advantages of DG methods,  which are suitable for the drift term if $\nabla \phi$ is large. For more information of the HDG methods for convection diffusion problems; see, e.g., \cite{Chen_Cockburn_Convection_Diffusion_IMAJNA_2012,Chen_Cockburn_Convection_Diffusion_MathComp_2014,Fu_Qiu_Zhang_Convection_Dominated_M2AN_2015,Qiu_Shi_Convection_Diffusion_JSC_2016,Chen_Li_Qiu_Posteriori_Convection_Diffusion_IMAJNA_2016}.

There are many different HDG schemes, see for example \cite{Cockburn_M_elasticity_IMA_2018,Cockburn_M_Stokes_IMA_2017,Cockburn_M_Diffusion_MathComp_2017,Cockburn_M_3D_M2AN_2017,Cockburn_M_2D_M2AN_2017,Lehrenfeld_PhD_thesis_2010,Cockburn_Gopalakrishnan_Lazarov_Unify_SINUM_2009}. Among all of these methods, two  are most popular, {following standard terminology} we call them are HDG$_k$ and HDG(A) in the rest of the paper. The HDG$_k$ method uses polynomials of degree $k$ to approximate the solution, the flux,  and the trace on the boundary face together with a positive stabilization parameter is chosen to be $\mathcal O(1)$. The HDG(A) method uses polynomial degree $k+1$ to approximate the solution, polynomial degree $k$ to approximate the flux and uses the so called Lehrenfeld-Sch\"oberl stabilization function, see \cite[Remark 1.2.4]{Lehrenfeld_PhD_thesis_2010}. These two methods were used to study the Poisson equation in \cite{Cockburn_Gopalakrishnan_Sayas_Porjection_MathComp_2010,Li_Xie_Family_JCAM_2016,Oikawa_Poisson_JSC_2015}, the linear elasticity \cite{Cockburn_M_elasticity_IMA_2018,Qiu_Shen_Shi_elasticity_MathComp_2018},  the convection diffusion equation in \cite{Chen_Cockburn_Convection_Diffusion_MathComp_2014,Chen_Cockburn_Convection_Diffusion_IMAJNA_2012,Qiu_Shi_Convection_Diffusion_JSC_2016}, the Stokes equation in \cite{Cockburn_Gopalakrishnan_Nguyen_Peraire_Sayas_Stokes_MathComp_2011,GongHuMateosSinglerZhang1} and the Navier-Stokes equation in \cite{Cesmelioglu_NS_MathComp_2017,Qiu_Shi_NS_IMAJNA_2016}.

{The goal of this paper is to design an HDG scheme by the appropriate choice
	of HDG spaces such that the overall scheme is optimally
	convergent and to prove semi-discrete optimal convergence rates  in $d$ spatial dimensions ($d=2,3$). The result is a new HDG scheme for the drift-diffusion system with attractive convergence properties.  We shall assume 
	that a suitably regular solution of the drift-diffusion system exists. For existence theory, see for example the book 
	of Markowich~\cite{MR821965}.}

{To develop our HDG method, we write the drift-diffusion system as a first order system by introducing new variable $\bm q$ and $\bm p$ such that $\bm q+\nabla u=0$, $\bm p+\nabla{\phi}=0$. Then \eqref{Drift_Diffusion_Main_Equation}, becomes the problem of finding $(u,{\bm q},\phi,{\bm p})$ such that}
\begin{subequations}\label{Drift_Diffusion_Main_Equation_Mixed_Weak_Form}
	\begin{align}
	\bm q+\nabla u&=0&\text{ in }\Omega\times (0,T],\label{Drift_Diffusion_Main_Equation_Mixed_Weak_Form_a}\\
	\bm p+\nabla {\phi}&=0&\text{ in }\Omega\times (0,T],\label{Drift_Diffusion_Main_Equation_Mixed_Weak_Form_b}\\
	u_t+ \nabla\cdot\bm q - \nabla\cdot(\bm p u)&=0&\text{ in }\Omega\times (0,T],\label{Drift_Diffusion_Main_Equation_Mixed_Weak_Form_c}\\
	\nabla\cdot\bm p+u&=0&\text{ in }\Omega\times (0,T],\label{Drift_Diffusion_Main_Equation_Mixed_Weak_Form_d}\\
	u&=g_u&\text{ on }\partial\Omega\times(0,T],\label{Drift_Diffusion_Main_Equation_Mixed_Weak_Form_e}\\
	{\phi}&=g_{{\phi}}&\text{ on }\partial\Omega\times(0,T],\label{Drift_Diffusion_Main_Equation_Mixed_Weak_Form_f}\\
	u(\cdot,0)&=u_0&\text{ in }\Omega.\label{Drift_Diffusion_Main_Equation_Mixed_Weak_Form_g}
	\end{align}
\end{subequations}
%\subsection{The HDG Formulation}
{We can now introduce our HDG formulation by first defining the mesh.} Let $\mathcal{T}_h$ denote a collection of disjoint  simplexes $K$ that partition $\Omega$ and let  $\partial \mathcal{T}_h$ be the set $\{\partial K: K\in \mathcal{T}_h\}$. {Here $h$ denotes the maximum diameter of the simplices in $\mathcal{T}_h$.  Since we will need to use an inverse inequality in our analysis, we assume that the mesh is shape regular and quasi-uniform.}

{We denote by
	$\mathcal{E}_h$ the set of all faces in the mesh. Then we define the set of interior and boundary faces (or edges when $d=2$) denoted $\mathcal{E}_h^o $ and $\mathcal{E}_h^\partial $ respectively. For each edge $e$ we say $e\in \mathcal{E}_h^o $ is} an interior face if the Lebesgue measure of  $e = \partial K^+ \cap \partial K^-$ {for some pair of elements $K^+,K^-\in\mathcal{T}_h$} is non-zero, similarly, $e \in \mathcal{E}_h^{\partial}$ is a boundary face if the Lebesgue measure of  $e = \partial K \cap \partial \Omega$ is non-zero. We  set
\begin{align*}
(w,v)_{\mathcal{T}_h} := \sum_{K\in\mathcal{T}_h} (w,v)_K,   \quad\quad\quad\quad\left\langle \zeta,\rho\right\rangle_{\partial\mathcal{T}_h} := \sum_{K\in\mathcal{T}_h} \left\langle \zeta,\rho\right\rangle_{\partial K},
\end{align*}
where $(\cdot,\cdot)_K$  denotes the $L^2(K)$ inner product and 
$ \langle \cdot, \cdot \rangle_{\partial K} $ denotes the $L^2$ inner product on $\partial K$.

{The HDG method uses discontinuous finite element spaces $\bm Q_h$, $ V_h$, $\widehat V_{h}$,  $\bm S_h$,  $\Psi_h$, $\widehat \Psi_{h}$ that we shall discuss shortly.  Assuming these are given, the approximate the solution of the mixed weak problem \eqref{Drift_Diffusion_Main_Equation_Mixed_Weak_Form} by} the HDG method seeks $(\bm q_h ,u_h ,\widehat u_h )\in \bm Q_h\times V_h\times\widehat V_{h}({g_u})$ and
$(\bm p_h ,{\phi}_h ,\widehat{\phi}_h )
\in \bm S_h\times \Psi_h\times\widehat \Psi_{h}(g_{\phi})$ satisfying
\begin{subequations}\label{Drift_Diffusion_HDG_Formulation}
	\begin{align}
	(\bm q_h, \bm r_1)_{\mathcal T_h} - (u_h,\nabla\cdot\bm r_1)_{\mathcal T_h} + \langle \widehat u_h, \bm r_1\cdot \bm n \rangle_{\partial \mathcal T_h}&=0,\label{Drift_Diffusion_HDG_Formulation_a}\\
	(\bm p_h, \bm r_2)_{\mathcal T_h} - (\phi_h,\nabla\cdot\bm r_2)_{\mathcal T_h} + \langle \widehat \phi_h, \bm r_1\cdot \bm n \rangle_{\partial \mathcal T_h}&=0,\label{Drift_Diffusion_HDG_Formulation_b}
	\end{align}
	for all $(\bm r_1, \bm r_2) \in \bm Q_h\times \bm S_h$, together with
	\begin{align}
	(u_{h,t}, w_1)_{\mathcal T_h} - (\bm q_h, \nabla w_1)_{\mathcal T_h} + \langle \widehat {\bm q}_h\cdot\bm n, w_1\rangle_{\partial \mathcal T_h} + (\bm p_h u_h, \nabla w_1)_{\mathcal T_h}&\nonumber \\
	- \langle \widehat {\bm p}_h\cdot\bm n \widehat u_h, w_1\rangle_{\partial \mathcal T_h}&=0,\label{Drift_Diffusion_HDG_Formulation_c}\\
	- (\bm p_h, \nabla w_2)_{\mathcal T_h} + \langle \widehat {\bm p}_h\cdot\bm n, w_2\rangle_{\partial \mathcal T_h} + (u_h,  w_2)_{\mathcal T_h} &=0\label{Drift_Diffusion_HDG_Formulation_d}
	\end{align}
	for all $(w_1,w_2)\in V_h\times \Psi_h$. The boundary fluxes must satisfy
	\begin{align}
	\langle \widehat{\bm q}_h\cdot\bm n,\mu_1\rangle_{\partial \mathcal T_h\backslash\partial \Omega}&=0,\label{Drift_Diffusion_HDG_Formulation_e}\\
	\langle \widehat{\bm p}_h\cdot\bm n,\mu_2\rangle_{\partial \mathcal T_h\backslash\partial \Omega} &=0\label{Drift_Diffusion_HDG_Formulation_f}
	\end{align}
	for all $(\mu_1,\mu_2)\in \widehat V_h(0)\times \widehat \Psi_h(0)$. The numerical fluxes $\widehat {\bm q}_h$ and $\widehat {\bm p}_h$ will be specified later.
\end{subequations}

As in \cite{Liu_LDG_DD_SCM_2016,Cesmelioglu_NS_MathComp_2017}, we shall need the following energy estimate
\begin{align}\label{important_argument}
\hspace{1em}&\hspace{-1em}\|\nabla u_h\|_{\mathcal T_h} + \| h_K^{-1/2}(u_h - \widehat u_h)\|_{\partial \mathcal T_h}^2 \nonumber\\
&\quad \le C\left(\|\bm q_h\|_{\mathcal T_h}^2 + \| h_K^{-1/2}(\Pi_k^\partial u_h - \widehat u_h)\|_{\partial \mathcal T_h}^2\right).
\end{align}
where $\Pi_k^\partial$ is a $L^2$ projection define in \eqref{L2_projection}.  Inequality \eqref{important_argument} cannot hold for the  HDG$_k$ method unless we take the stabilization function to be $h_K^{-1}$. However, in this case we only have a suboptimal convergence rate for the flux $\bm q $. Hence we need to use the HDG(A) method to approximate the equation \eqref{Drift_Diffusion_Main_Equation_a}, i.e., we choose  
\begin{align*}
\bm Q_h&:=\{\bm v_h\in [L^2(\Omega)]^d:\bm v_h|_K\in [\mathcal{P}^{k}(K)]^{d},\forall K\in\mathcal{T}_h\},\\
V_h&:=\{v_h\in L^2(\Omega):v_h|_K\in \mathcal{P}^{k+1}(K),\forall K\in\mathcal{T}_h\},\\
\widehat V_h(g)&:=\{\widehat{v}_h\in L^2(\mathcal{E}_h):\widehat v_h|_E\in \mathcal{P}^{k}(\mathcal{E}_h),\forall E\in\mathcal{E}_h, \widehat v_h|_{\mathcal{E}_h^{\partial}}=\Pi_k^{\partial}g\},
\end{align*}
where $\mathcal{P}^k(K)$ denotes the set of polynomials of degree at most $k$ on the element $K$
{(similarly $ \mathcal{P}^{k}(\mathcal{E}_h)$ denotes the set of polynomials of degree at most $k$ on the 
	faces in the mesh).}
Moreover, the  numerical trace of the flux on $\partial \mathcal T_h$ is defined as
\begin{align}
\widehat{\bm q}_h \cdot\bm n = {\bm q}_h \cdot\bm n  + h_K^{-1}(\Pi_k^\partial u_h - \widehat u_h),\label{Drift_Diffusion_HDG_Formulation_g}
\end{align}
{where $\Pi_k^\partial$ denotes $L^2$ projection onto $ \mathcal{P}^{k}(\mathcal{E}_h)$ which can be done face by face.}

To avoid  a reduction in the convergence rate for the solution $u_h$, the polynomial degree of the space $V_h$ for $u_h$ and the space $\bm S_h$ for $\bm p_h$ need to be the same, i.e.,
\begin{align*}
\bm S_h:=\{\bm v_h\in [L^2(\Omega)]^d:\bm v_h|_K\in [\mathcal{P}^{k+1}(K)]^{d},\forall K\in\mathcal{T}_h\}.
\end{align*}
If we choose the HDG(A) method to discretize \eqref{Drift_Diffusion_Main_Equation_b} we would need to use polynomials of degree $k+2$ to approximate $\phi$, but in this case, we get a suboptimal convergence rate for $\phi$. Therefore, we use HDG$_{k+1}$ to discretize \eqref{Drift_Diffusion_Main_Equation_b} and so choose
\begin{align*}
\Psi_h&:=\{v_h\in L^2(\Omega):v_h|_K\in \mathcal{P}^{k+1}(K),\forall K\in\mathcal{T}_h\},\\
\widehat \Psi_h(g)&:=\{\widehat{v}_h\in L^2(\mathcal{E}_h):\widehat v_h|_E\in \mathcal{P}^{k+1}(\mathcal{E}_h),\forall E\in\mathcal{E}_h, \widehat v_h|_{\mathcal{E}_h^{\partial}}=\Pi_{k+1}^{\partial}g\}.
\end{align*}
and the  numerical trace of the flux on $\partial \mathcal T_h$ is defined as
\begin{align}
\widehat{\bm p}_h \cdot\bm n = {\bm p}_h \cdot\bm n  + \tau(\phi_h - \widehat \phi_h),\label{Drift_Diffusion_HDG_Formulation_h}
\end{align}
where $\tau$ is a positive $\mathcal O(1)$ function and the initial condition $u_h(0)$ will be specifically in \Cref{HDG_elliptic_projection_and_basic_estimates}. {If needed, $\tau$ can be chosen to provide upwind
	stabilization as in \cite{Qiu_Shi_Convection_Diffusion_JSC_2016}.}

The organization of the paper is as follows. In \Cref{error_analysis}, we present our main results and some useful projections. Then   the proof of the main results  is given in \Cref{Proofofmain_theorem}. In \Cref{Numerical_Results}, we provide some numerical experiments to support our theoretical results.

{In this paper we denote by $\Vert\cdot\Vert_{s,D}$ the $H^s(D)$ Sobolev norm. As we have already done, bold face quantities denote vectors. If $s$ is not present, the $L^2$ norm is assumed so that, for example,
	$
	\Vert w\Vert_{{\cal T}_h}=\sqrt{(w,w)_{{\cal T}_h}}.
	$}

\section{Main result and preliminary material}
\label{error_analysis}

In this section, we first present the main result in \Cref{main_result} for the semidiscrete HDG formulation \eqref{Drift_Diffusion_HDG_Formulation}. Next, we  provide preliminary material in \Cref{Preliminary_material},  which are required for the analysis. 

We use the standard notation $W^{m,p}(D)$ for Sobolev spaces on $D$ with norm $\|\cdot\|_{m,p,D}$ and seminorm $|\cdot|{m,p,D}$. We also write $H^m(D)$ instead of $W^{m,2}(D)$, and we omit the index $p$ in the corresponding norms and seminorms. Moreover, we omit the index $m$ when $m=0$.

Throughout, we assume the data and the solution of \eqref{Drift_Diffusion_Main_Equation}  are smooth enough for our analysis.

\subsection{Main result}
\label{main_result}
{The proof of our main error estimate relies on the use of  duality arguments and requires sufficient regularity for the solution of the corresponding problem.  In particular:}
%	\begin{enumerate} \item \PM{For
%	for  $\Theta\in L^2(\Omega)$, let  $(\bm\Phi,\Psi)$  be the solution of 
%	\begin{equation}\label{Dual_PDE}
%	\begin{split}
%	\bm{\Phi}+\nabla\Psi &= 0\qquad\qquad\text{in}\ \Omega,\\
%	\nabla\cdot\bm \Phi &= \Theta\qquad\quad~~\text{in}\ \Omega,\\
%	\Psi &= 0\qquad\qquad\text{on}\ \partial\Omega.
%	\end{split}
%	\end{equation}
%	We assume the solution $(\bm\Phi,\Psi)$  have the  following regularity
%	\begin{align}\label{reg_e}
%	\|{\bm\Phi}\|_{H^1(\Omega)} + \|{\Psi}\|_{H^2(\Omega)} \le C_{\textup{reg}} \|{\Theta}\|_{\mathcal T_h}.
%	\end{align}
%	It is well known that the above regularity holds if the domain is convex, which is usually the case in solar cell
%	applications.}
%	\item 
\begin{assumption} {Let $\bm p\in H^1((0,T),\bm W_1^\infty(\Omega))$ denote a given vector function of position and time.  Let $M>0$ such that for all time $t\in(0,T)$
		\begin{align}
		M\ge \|\nabla\cdot\bm p(t)\|_{0,\infty}+2\|\partial_t\bm p(t)\|_{0,\infty}.\label{Mge}
		\end{align}
		If ${\bm p}=0$, set $M=0$.
		Then, for  $\Theta\in L^2(\Omega\times(0,T))$,  let $(\bm\Psi,\Phi)$  be the solution of 
		\begin{align}
		\begin{split}
		\bm{\Psi} + \nabla\Phi&=0  \ \qquad \textup{ in }\Omega,\\
		M\Phi+\nabla\cdot\bm\Psi +  \bm p  \cdot\nabla \Phi
		&=\Theta \qquad \textup{ in }\Omega,\\
		\Phi&=0\qquad \ \  \textup{on } \partial\Omega.
		\end{split}\label{D2}
		\end{align}
		We assume the solution $(\bm\Psi,\Phi)$ has the  following regularity
		\begin{align}\label{elliptic_regularity}
		\|\bm\Psi\|_{H^1(\Omega)}+\|\Phi\|_{H^2(\Omega)}&\le C_{\text{reg}}\|\Theta\|_{\mathcal T_h}.
		\end{align}
		It is well known that the above regularity holds if the domain is convex, which is usually the case in solar cell
		applications.}
\end{assumption}	

%	\end{enumerate}

We can now state our main result for the HDG method.
%\marginpar{Can we be specific about the
%regularity assumed for $(\bm q, u, \bm p, \phi)\in H^2((0,T); \bm H^{k+2}(\Omega))\times H^2((0,T); \bm H^{k+3}(\Omega))\times H^2((0,T); \bm H^{k+3}(\Omega))\times H^2((0,T); \bm H^{k+3}(\Omega))$?}
\begin{theorem}\label{main_theorem}
	{Assume that  (\ref{elliptic_regularity}) hold and that the mesh is quasi-uniform.}
	{Let 
		\begin{align*}
		(\bm q, u)& \in H^2((0,T),\bm H^{k+1}(\Omega))\times H^2((0,T),  H^{k+2}(\Omega)), \\
		(\bm p, \phi)& \in H^2((0,T),\bm H^{k+2}(\Omega))\times H^2((0,T), H^{k+3}(\Omega))
		\end{align*}
		solve}  \eqref{Drift_Diffusion_Main_Equation_Mixed_Weak_Form} and let $(\bm q_h, u_h, \bm p_h, \phi_h)
	{\in\bm Q_h\times V_h\times \bm S_h\times \Psi_h}$ be the solution of 
	{ the semi-discrete HDG equations }\eqref{Drift_Diffusion_HDG_Formulation}. Then we have 
	\begin{align*}
	\|  u- u_h\|_{\mathcal T_h} + \| \phi- \phi_h\|_{\mathcal T_h}  + \| \bm p- \bm p_h\|_{\mathcal T_h}  \le  C  h^{k+2}
	\end{align*}
	for all $t\in[0,T]$, and 
	\begin{align*}
	\sqrt{\int_0^T \| \bm q - \bm q_h\|_{\mathcal T_h}^2 dt } \le  C  h^{k+1}.
	\end{align*}
\end{theorem}

\begin{remark}
	The error estimates in \Cref{main_theorem} are optimal for the variables $\bm q$, $u$, $\bm p$ and $\phi$. Since the global degrees of freedom are the numerical traces, then from the point of view of global degrees of freedom, the error estimates for the variable $u$ is superconvergent, {which, to our knowledge,}  is the first time this has been proved in the literature.
\end{remark}

\subsection{Preliminary material}
\label{Preliminary_material}
We first introduce the HDG$_k$ projection operator $\Pi_h(\bm{p},\phi) := (\bm{\Pi}_{V} \bm{p},\Pi_{W}\phi)$ defined in \cite{Cockburn_Gopalakrishnan_Sayas_Porjection_MathComp_2010}, where
$\bm{\Pi}_{V} \bm{p}$ and $\Pi_{W}\phi$ denote components of the projection of $\bm{p}$ and $\phi$ into $\bm{S}_h$ and $\Psi_h$, respectively.
For each element $K\in\mathcal T_h$, the projection  is determined by the equations
\begin{subequations}\label{HDG_projection_operator}
	\begin{align}
	(\bm\Pi_V\bm p,\bm r)_K &= (\bm p,\bm r)_K,\qquad\qquad \forall \bm r\in[\mathcal P_{k}(K)]^d,\label{projection_operator_1}\\
	(\Pi_W\phi, w)_K &= (\phi, w)_K,\qquad\qquad \forall  w\in \mathcal P_{k}(K	),\label{projection_operator_2}\\
	\langle\bm\Pi_V\bm p\cdot\bm n+\tau\Pi_W\phi,\mu\rangle_{e} &= \langle\bm p\cdot\bm n+\tau \phi,\mu\rangle_{e},~\;\forall \mu\in \mathcal P_{k+1}(e)\label{projection_operator_3}
	\end{align}
\end{subequations}
for all faces $e$ of the simplex $K$.   The approximation properties of the HDG$_k$ projection \eqref{HDG_projection_operator} are given in the following result from \cite{Cockburn_Gopalakrishnan_Sayas_Porjection_MathComp_2010}:
\begin{lemma}\label{pro_error}
	Suppose $k\ge 0$, $\tau|_{\partial K}$ is nonnegative and $\tau_K^{\max}:=\max\tau|_{\partial K}>0$. Then the system \eqref{HDG_projection_operator} is uniquely solvable for $\bm{\Pi}_V\bm{p}$ and $\Pi_W \phi$. Furthermore, there is a constant $C$ independent of $K$ and $\tau$ such that
	\begin{subequations}
		\begin{align}
		\|{\bm{\Pi}_V}\bm{p}-\bm p\|_K &\leq Ch_{K}^{\ell_{\bm{p}}+1}|\bm{p}|_{\ell_{\bm{p}}+1,K}+Ch_{K}^{\ell_{{\phi}}+1}\tau_{K}^{*}{|\phi|}_{\ell_{{\phi}}+1,K},\label{Proerr_q}\\
		\|{{\Pi}_W}{\phi}-\phi\|_K &\leq Ch_{K}^{\ell_{{\phi}}+1}|{\phi}|_{\ell_{{\phi}}+1,K}+C\frac{h_{K}^{\ell_{{\bm{p}}}+1}}{\tau_K^{\max}}{|\nabla\cdot \bm{p}|}_{\ell_{\bm{p}},K}\label{Proerr_u}
		\end{align}
	\end{subequations}
	for $\ell_{\bm{p}},\ell_{\phi}$ in $[0,k+1]$. Here $\tau_K^{*}:=\max\tau|_{{\partial K}\backslash e^{*}}$, where $e^{*}$ is a face of  $K$ at which $\tau|_{\partial K}$ is maximum.
\end{lemma}

We next define the standard $L^2$ projections $\bm\Pi_{k}^o :  [L^2(\Omega)]^d \to \bm Q_h$, $\Pi_{k+1}^o :  L^2(\Omega) \to V_h$, and $\Pi_k^{\partial}:  L^2(\mathcal E_h) \to \widehat V_h$, which satisfy
\begin{equation}\label{L2_projection}
\begin{split}
(\bm\Pi_k^o \bm q,\bm r_1)_{K} &= (\bm q,\bm r_1)_{K} ,\qquad \forall \bm r_1\in [{\mathcal P}_{k}(K)]^d,\\
(\Pi_{k+1}^o u,w_1)_{K}  &= (u,w_1)_{K} ,\qquad \forall w_1\in \mathcal P_{k+1}(K),\\
\langle \Pi_k^\partial u, \mu_1\rangle_{ e} &= \left\langle  u, \mu_1\right\rangle_{e }, \qquad\;\;\; \forall \mu_1\in \mathcal P_{k}(e).
\end{split}
\end{equation}
In the analysis, we use the following classical results \cite[Lemma 3.3]{ChenSinglerZhang1}:
\begin{subequations}\label{classical_ine}
	\begin{align}
	\|{\bm q -\bm\Pi_k^o \bm q}\|_{\mathcal T_h} &\le  C h^{k+1} \|{\bm q}\|_{k+1,\Omega},\quad\| {u -{\Pi_{k+1}^o u}}\|_{\mathcal T_h} \le  C h^{k+2} \|{u}\|_{k+2,\Omega},\label{eq27a}\\
	\| {u - {\Pi_{k+1}^o u}}\|_{\partial\mathcal T_h} &\le  C h^{k+\frac 3 2} \|{u}\|_{k+2,\Omega},
	\quad \| {w}\|_{\partial \mathcal T_h} \le  C h^{-\frac 12} \| {w}\|_{ \mathcal T_h}, \: \forall w\in V_h.\label{eq27b}
	\end{align}
\end{subequations}

To shorten lengthy equations, we rewrite the HDG formulation \eqref{Drift_Diffusion_HDG_Formulation} in the following compact form: find 
$(\bm q_h ,u_h ,\widehat u_h )\in \bm Q_h\times V_h\times\widehat V_{h}(g_u)$ and
$(\bm p_h ,{\phi}_h ,\widehat{\phi}_h )
\in \bm S_h\times \Psi_h\times\widehat \Psi_{h}(g_{\phi})$ such that
\begin{subequations}\label{Drift_Diffusion_HDG_Formulation_Compact}
	\begin{align}
	(\partial_tu_h ,w_1)_{\mathcal{T}_h}+\mathscr A(\bm q_h ,u_h ,\widehat u_h;\bm r_1,w_1,\mu_1)+\mathscr C(\bm p_h,\widehat {\bm p}_h; u_h ,\widehat u_h;w_1)=0,\label{Drift_Diffusion_HDG_Formulation_Compact_first}\\
	{\mathscr{B}}(\bm p_h ,{\phi}_h ,\widehat{\phi}_h ;\bm r_2,w_2,\mu_2)+( u_h,w_2)_{\mathcal{T}_h}=0,\label{Drift_Diffusion_HDG_Formulation_Compact_second}
	\end{align}
	for all
	$(\bm r_1, \bm r_2, w_1,w_2,\mu_1,\mu_2)\in \bm Q_h\times \bm S_h\times V_h\times\Psi_h\times \widehat V_h(0)\times \widehat \Psi_h(0)$, where the HDG {bilinear forms ${\mathscr {A}}$, ${\mathscr {B}}$ and  the trilinear form $\mathscr C$ are defined by}
	\begin{align}\label{def_A}
	\begin{split}
	\hspace{1em}&\hspace{-1em}{\mathscr{A}}(\bm q_h,u_h,\widehat u_h;\bm r_1, w_1, \mu_1)\\
	&=(\bm q_h,\bm r_1)_{\mathcal{T}_h}  - (u_h,\nabla\cdot\bm r_1)_{\mathcal{T}_h}
	+\langle\widehat u_h,\bm r_1\cdot \bm n
	\rangle_{\partial\mathcal{T}_h}+(\nabla\cdot \bm q_h,w_1)_{\mathcal{T}_h}\\
	&\quad-\langle \bm q_h \cdot\bm n,\mu_1 \rangle_{\partial\mathcal{T}_h}+\langle h_{K}^{-1}(\Pi_{k}^{\partial}u_h-\widehat u_h),\Pi_{k}^{\partial} w_1-\mu_1\rangle_{\partial\mathcal{T}_h}
	\end{split}
	\end{align}
	for all $(\bm q_h ,u_h ,\widehat u_h, \bm r_1, w_1,\mu_1)\in \bm Q_h\times V_h\times\widehat V_{h}(g_u)\times \bm Q_h\times V_h\times\widehat V_{h}(0)$,
	\begin{align}\label{def_B}
	\begin{split}
	\hspace{1em}&\hspace{-1em} {\mathscr {B}}(\bm p_h, \phi_h,\widehat {\phi}_h; \bm r_2, w_2,\mu_2)\\
	&=(\bm p_h,\bm r_2)_{\mathcal{T}_h}- ({\phi_h},\nabla\cdot\bm r_2)_{\mathcal{T}_h}
	+\langle\widehat {\phi}_h,\bm r_2\cdot \bm n
	\rangle_{\partial\mathcal{T}_h}+(\nabla\cdot \bm p_h,w_2)_{\mathcal{T}_h}\\
	&\quad-\langle
	\bm p_h\cdot\bm n,\mu_2\rangle_{\partial\mathcal{T}_h}
	+\langle \tau({\phi}_h-\widehat {\phi}_h),{w_2}-\mu_2
	\rangle_{\partial\mathcal{T}_h}
	\end{split}
	\end{align}
	for all $(\bm p_h ,\phi_h ,\widehat \phi_h, \bm r_2, w_2,\mu_2)\in \bm S_h\times \Psi_h\times\widehat \Psi_{h}(g_{\phi})\times \bm S_h\times \Psi_h\times\widehat \Psi_{h}(0)$,
	\begin{align}\label{def_C}
	\hspace{1em}&\hspace{-1em} \mathscr C(\bm p,\widehat {\bm p};u_h,\widehat u_h;w_1)= ( \bm pu_h, \nabla w_1)_{\mathcal{T}_h} - \langle  \widehat {\bm p} \cdot\bm n \widehat u_h, w_1 \rangle_{\partial\mathcal{T}_h}
	\end{align}
	for all $(u_h, \widehat u_h, w_1,\mu_1)\in V_h\times\widehat V_{h}(g_{u})\times V_h\times\widehat V_{h}(0)$.
\end{subequations}

Next, we present basic properties of the operators $\mathscr A$ and $\mathscr B$.
\begin{lemma}\label{basis_property1}
	For any	$(\bm q_h ,u_h ,\widehat u_h, \bm r_1, w_1,\mu_1)\in \bm Q_h\times V_h\times\widehat V_{h}(0)\times \bm Q_h\times V_h\times\widehat V_{h}(0)$ and  $(\bm p_h ,\phi_h ,\widehat \phi_h, \\ \bm r_2, w_2,\mu_2)\in \bm S_h\times \Psi_h\times\widehat \Psi_{h}(0)\times \bm S_h\times \Psi_h\times\widehat \Psi_{h}(0)$, we have 
	\begin{align*}
	\mathscr A(\bm q_h ,u_h ,\widehat u_h; -\bm r_1, w_1,\mu_1)  =	\mathscr A( \bm r_1, w_1,\mu_1;-\bm q_h,u_h ,\widehat u_h), \\
	\mathscr B(\bm p_h ,\phi_h ,\widehat \phi_h; -\bm r_2, w_2,\mu_2)  =	\mathscr B( \bm r_2, w_2,\mu_2;-\bm p_h,\phi_h ,\widehat \phi_h), 
	\end{align*}
	and 
	\begin{align*}
	\mathscr A(\bm q_h ,u_h ,\widehat u_h; \bm q_h,u_h,\widehat u_h) =\|\bm q_h\|_{\mathcal T_h}^2 + \|h_K^{-1/2} (\Pi_k^\partial u_h - \widehat u_h)\|_{\partial \mathcal T_h}^2,\\
	\mathscr B(\bm p_h, \phi_h ,\widehat {\phi}_h;\bm p_h, \phi_h,\widehat {\phi}_h) =\|\bm p_h\|_{\mathcal T_h}^2 + \|\sqrt{\tau}(\phi_h - \widehat \phi_h)\|_{\partial \mathcal T_h}^2. 	
	\end{align*}
\end{lemma}
The proof of \Cref{basis_property1} is  straightforward, hence we omit it here.

The proof of the  following two lemmas are found in \cite[Lemma 3.2 ]{Qiu_Shi_Convection_Diffusion_JSC_2016} and \cite[Equation (1.3)]{Brenner_Poincare_SINUM_2003}, respectively .
\begin{lemma} \label{energy_argument1}
	If  $(\bm q_h, u_h, \widehat u_h)$ satisfies the equation \eqref{Drift_Diffusion_HDG_Formulation_a}, then we have 
	\begin{align*}
	\|\nabla u_h\|_{\mathcal{T}_h}
	+\|h_K^{-1/2}(u_h-\widehat u_h)\|_{\partial\mathcal{T}_h}
	\le C 	\left(\|\bm q_h\|_{\mathcal{T}_h}+\|h_K^{-1/2}(\Pi_{k}^{\partial} u_h-\widehat u_h)\|_{\partial\mathcal{T}_h}
	\right).
	\end{align*}
\end{lemma}

\begin{lemma}[Piecewise Poinc\'are-Friedrichs' inequality]\label{Poincare}  
	Let $v_h\in H^1(\mathcal{T}_h)$, then we have
	\begin{align*}
	\|v_h\|_{\mathcal{T}_h}^2 \le C\left(\|\nabla v_h\|_{\mathcal{T}_h}^2
	+|\langle v_h,1 \rangle_{\partial\Omega}|^2+\sum_{e\in\mathcal{E}^o_h} |e|^{d/(1-d)}\left(\int_e[\![v_h]\!]\,ds\right)^2 \right),
	\end{align*}
	{where $|e|$ denotes the measure of $e$.}
\end{lemma}

By \Cref{Poincare},  we immediately have the following lemma.%\marginpar{\PM{Needs more detail - provide brief proof}}
\begin{lemma}[HDG Poin\'care inequality]\label{HDG_Poincare}
	If $(v_h,\widehat v_h)\in {V}_h\times \widehat V_h(0)$, then we have
	\begin{align*}
	\|v_h\|^2_{\mathcal{T}_h}\le C
	\left(
	\|\nabla v_h\|^2_{\mathcal{T}_h}+
	\|h_K^{-1/2}(\Pi_k^{\partial}v_h-\widehat v_h)\|^2_{\partial\mathcal{T}_h}
	\right).
	\end{align*}		
\end{lemma}
\begin{proof}
	By \Cref{HDG_Poincare}, $\widehat v_h$ is zero on $\partial \Omega$ and is single valued on interior faces. We have 
	\begin{align*}
	\|v_h\|^2_{\mathcal{T}_h}&\le C
	\left(
	\|\nabla v_h\|^2_{\mathcal{T}_h}+
	\|h_K^{-1/2}[\![v_h]\!] \|^2_{\mathcal E_h}
	\right)\\
	&= C
	\left(
	\|\nabla v_h\|^2_{\mathcal{T}_h}+
	\|h_K^{-1/2}[\![v_h-\Pi_k^\partial v_h +\Pi_k^\partial v_h -\widehat v_h]\!] \|^2_{\mathcal E_h}
	\right)\\
	&\le  C
	\left(
	\|\nabla v_h\|^2_{\mathcal{T}_h}+
	\|h_K^{-1/2}(v_h-\Pi_k^\partial v_h)\|_{\partial\mathcal{T}_h}^2 +\|h_K^{-1/2}(\Pi_k^\partial v_h -\widehat v_h)\|^2_{\partial\mathcal{T}_h}
	\right)\\
	&\le C
	\left(
	\|\nabla v_h\|^2_{\mathcal{T}_h}+
	\|h_K^{-1/2}(\Pi_k^{\partial}v_h-\widehat v_h)\|^2_{\partial\mathcal{T}_h}
	\right).
	\end{align*}
\end{proof}

\section{Proof of \Cref{main_theorem}.}
\label{Proofofmain_theorem}
To prove \Cref{main_theorem}, we follow a similar strategy to that in \cite{ChenCockburnSinglerZhang1}. We first bound the error between the solution of an HDG elliptic projection {defined in} \eqref{projection01} and the solution of the system \eqref{Drift_Diffusion_Main_Equation_a}. Then we bound the error between the solution of the HDG elliptic projection \eqref{projection01} and the HDG formulation \eqref{Drift_Diffusion_HDG_Formulation_Compact_first} and the error between the solution of the system \eqref{Drift_Diffusion_Main_Equation_b} and the solution of the HDG formulation \eqref{Drift_Diffusion_HDG_Formulation_Compact_second}. A simple application of the triangle inequality then gives a bound on the error between the solution of the HDG formulation \eqref{Drift_Diffusion_HDG_Formulation_Compact} and the system \eqref{Drift_Diffusion_Main_Equation}. First, we present the HDG elliptic projection.

\subsection{HDG elliptic projection and basic estimates}
\label{HDG_elliptic_projection_and_basic_estimates}

For $t\in [0,T]$, let $({\bm q}_{Ih},u_{Ih}, \widehat u_{Ih})\in \bm Q_h\times V_h\times \widehat V_h({g_u})$ be the solution of 
\begin{align}\label{projection01}
\begin{split}
\hspace{1em}&\hspace{-1em} M(u_{Ih},w_1)_{\mathcal{T}_h}+{\mathscr {A}}(\bm q_{Ih},u_{Ih},\widehat u_{Ih}; r_1,w_1,\mu_1) + \mathscr C(\bm p,\bm p; u_{Ih},\widehat u_{Ih};w_1)\\
&=(Mu-u_t,w_1)_{\mathcal{T}_h}
\end{split}
\end{align}
for all $(r_1,w_1,\mu_1)\in \bm Q_h\times V_h\times \widehat V_h(0)$ {where}  $M$ is a given constant such that (\ref{Mge}) holds.

Take the partial derivative of \eqref{projection01} with respect to $t$, hence, $(\partial_t \bm q_{Ih}, \partial_t u_{Ih},\\ \partial_t \widehat u_{Ih})\in \bm Q_h\times V_h\times \widehat V_h(\partial_t {g_u})$ is the solution of
\begin{align}\label{projection02}
\begin{split}
\hspace{1em}&\hspace{-1em} M(\partial_t u_{Ih},w_1)_{\mathcal{T}_h}+{\mathscr {A}}(\bm \partial_t{\bm q}_{Ih},\partial_tu_{Ih},\partial_t\widehat u_{Ih}; r_1,w_1,\mu_1) \\
&\quad + \mathscr C(\partial_t\bm p,\partial_t\bm p; u_{Ih},\widehat u_{Ih};w_1) +\mathscr C(\bm p,\bm p; \partial_t u_{Ih}, \partial_t \widehat u_{Ih};w_1)  \\
&=(Mu_t-u_{tt},w_1)_{\mathcal{T}_h}
\end{split}
\end{align}
for all $(r_1,w_1,\mu_1)\in \bm Q_h\times V_h\times \widehat V_h({0})$.

We choose the initial condition $u_h(0) = u_{Ih}(0)$ for the {purposes of analysis. In fact,} the initial condition $u_h(0)$ can be chosen to be the  $L^2$ projection of $u_0$, i.e., $\Pi_k^o u_0$. 

The following result,  \Cref{HDG_elliptic_projection_approximation}, gives the error between the solution of an HDG elliptic projection \eqref{projection01} and the solution of the system \eqref{Drift_Diffusion_Main_Equation_a} and the proofs are given in \Cref{Appendix}.

\begin{theorem}\label{HDG_elliptic_projection_approximation} 
	For any $t\in [0,T]$, if the elliptic regularity inequality \eqref{elliptic_regularity} holds and $h$ is small enough, then we have the following error estimates
	\begin{subequations}
		\begin{align}
		\|u-u_{Ih}\|_{\mathcal{T}_h}&\le Ch^{k+2}\|u\|_{k+2},\label{HDG_elliptic_projection_approximation1}\\
		\|\bm q-\bm q_{Ih}\|_{\mathcal{T}_h}+\|h_K^{-1/2}(\Pi_k^{\partial}u_{Ih}-\widehat u_{Ih})\|_{\partial\mathcal{T}_h}&\le Ch^{k+1}\|u\|_{k+2}\label{HDG_elliptic_projection_approximation2}.
		\end{align}
		In addition, we have 
		\begin{align}
		\|\partial_tu-\partial_tu_{Ih}\|_{\mathcal{T}_h}&\le Ch^{k+2 }\|\partial_t u\|_{k+2}.\label{HDG_elliptic_projection_approximation3}
		\end{align}
	\end{subequations}
\end{theorem}

\subsection{Error equation between the HDG formulation \eqref{Drift_Diffusion_HDG_Formulation_Compact} and the HDG elliptic projection \eqref{projection01}.}
To bound the error between the solution of the HDG elliptic projection \eqref{projection01} and the system \eqref{Drift_Diffusion_HDG_Formulation_Compact_first}, and the error between the solution of the HDG formulation \eqref{Drift_Diffusion_HDG_Formulation_Compact_second} and the system \eqref{Drift_Diffusion_Main_Equation_b}. We first derive the error equation summarized in the next lemma. To simplify notation, we define
\begin{align*}
&\xi_h^{\bm q}=\bm q_{Ih}-\bm q_h,\ \ \ \ \ \ \
\xi_h^{u}=u_{Ih}-u_h,\ \ \ \ \ \ \
\xi_h^{\widehat u}=\widehat u_{Ih}-\widehat u_{h},\\
&\xi_h^{\bm p}=\bm{\Pi}_{V}\bm p-\bm p_h,\ \ \ \ \
\xi_h^{{\phi}}=\Pi_{W}{\phi}-{\phi}_h,\ \ \ \ 
\xi_h^{\widehat {\phi}}=\Pi_{k+1}^{\partial}{\phi}-\widehat {\phi}_{h}.
\end{align*}

\begin{lemma}\label{error_equation}
	For any	$(\bm r_1, w_1,\mu_1, \bm r_2, w_2,\mu_2)\in \bm Q_h\times V_h\times\widehat V_{h}(0)\times \bm S_h\times \Psi_h\times\widehat \Psi_{h}(0)$, we have the following error equation
	\begin{subequations}
		\begin{align}
		\hspace{1em}&\hspace{-1em}(\partial_t \xi_h^u,w_1)_{\mathcal{T}_h}+ \mathscr A
		(\xi_h^{\bm q},\xi_h^u,\xi_h^{\widehat u};\bm r_1,w_1,\mu_1)\nonumber\\
		&=(\partial_t(u_{Ih}-u),w_1)_{\mathcal{T}_h}+M(u-u_{Ih},w_1)_{\mathcal{T}_h}\nonumber\\
		&\quad- {\mathscr {C}} (\bm p,\bm p;\xi_h^u,\xi_h^{\widehat u};w_1) -  {\mathscr {C}} (\bm p - \bm p_h, \bm p-\widehat {\bm p}_h;u_h,\widehat u_h;w_1),\label{error_equation1}\\
		\hspace{1em}&\hspace{-1em} 	
		\mathscr B(\xi_h^{\bm p},\xi_h^{{\phi}},\xi_h^{\widehat {\phi}};
		\bm r_2,w_2,\mu_2)=(\bm\Pi_V \bm p - \bm p,\bm r_2)_{\mathcal{T}_h} - ( u - u_h, w_2)_{\mathcal T_h}.\label{error_equation2}
		\end{align}
	\end{subequations}
\end{lemma}
\begin{proof} 
	We first prove \eqref{error_equation1}. Subtracting equation \eqref{Drift_Diffusion_HDG_Formulation_Compact_first} from \eqref{projection01} and using the definition of $\mathscr A$ and $\mathscr C$ we get 
	\begin{align*}
	&\quad M(u_{Ih}, w_1)_{\mathcal T_h}+{\mathscr {A}}
	(\xi_h^{\bm q},\xi_h^u,\xi_h^{\widehat u};\bm r_1,w_1, \mu_1) + {\mathscr {C}} (\bm p,\bm p;u_{Ih} ,\widehat u_{Ih};w_1) \\
	&\quad  - (\partial_t u_h,w_1)_{\mathcal{T}_h}- \mathscr C(\bm p_h, \widehat {\bm p}_h; u_{h} ,\widehat u_{h};w_1)\\
	& = (Mu- u_t,w_1)_{\mathcal{T}_h}.
	\end{align*}
	This gives
	\begin{align*}
	\hspace{1em}&\hspace{-1em}(\partial_t\xi_h^u,w_1)_{\mathcal{T}_h}+{\mathscr {A}}
	(\xi_h^{\bm q},\xi_h^u,\xi_h^{\widehat u};\bm r_1,w_1, \mu_1)\\
	&\quad 
	+{\mathscr {C}} (\bm p,\bm p;u_h,\widehat u_h;w_1)- \mathscr C(\bm p_h,\widehat {\bm p}_h; u_{Ih} ,\widehat u_{Ih};w_1)\\
	&=(\partial_tu_{Ih},w_1)_{\mathcal{T}_h}-(u_t,w_1)_{\mathcal{T}_h}
	+M(u-u_{Ih},w_1)_{\mathcal{T}_h}.
	\end{align*}
	We note that the nonlinear operator $\mathscr C$ is linear for each variables, hence we have
	\begin{align*}
	\hspace{1em}&\hspace{-1em} \mathscr C(\bm p,\bm p; u_{Ih} ,\widehat u_{Ih};w_1) - {\mathscr {C}} (\bm p_h,\widehat {\bm p}_h;u_h,\widehat u_h;w_1)\\
	&=\mathscr C(\bm p,\bm p; u_{Ih} ,\widehat u_{Ih};w_1) -  \mathscr C(\bm p,\bm p; u_{h} ,\widehat u_{h};w_1)  \\
	&\quad +{\mathscr {C}} (\bm p,\bm p;u_h,\widehat u_h;w_1) - {\mathscr {C}} (\bm p_h,\widehat {\bm p}_h;u_h,\widehat u_h;w_1)\\
	& =  {\mathscr {C}} (\bm p,\bm p;\xi_h^u,\xi_h^{\widehat u};w_1) +  {\mathscr {C}} (\bm p - \bm p_h, \bm p-\widehat {\bm p}_h;u_h,\widehat u_h;w_1).
	\end{align*}
	This implies
	\begin{align*}
	\hspace{1em}&\hspace{-1em}(\partial_t \xi_h^u,w_1)_{\mathcal{T}_h}+ \mathscr A
	(\xi_h^{\bm q},\xi_h^u,\xi_h^{\widehat u};\bm r_1,w_1,\mu_1)\nonumber\\
	&=(\partial_t(u_{Ih}-u),w_1)_{\mathcal{T}_h}+M(u-u_{Ih},w_1)_{\mathcal{T}_h}\nonumber\\
	&\quad  - {\mathscr {C}} (\bm p,\bm p;\xi_h^u,\xi_h^{\widehat u};w_1) -  {\mathscr {C}} (\bm p - \bm p_h, \bm p-\widehat {\bm p}_h;u_h,\widehat u_h;w_1).
	\end{align*}
	Next, we prove \eqref{error_equation2}. By the definition of $\mathscr  B$ in \eqref{def_B}, we have
	\begin{align*}
	\hspace{1em}&\hspace{-1em} \mathscr B(\bm{\Pi}_{V}\bm p,\Pi_{W} \phi,\Pi_{k+1}^{\partial}{\phi}; \bm r_2,w_2,\mu_2)\\
	&=(\bm{\Pi}_{V}\bm p,\bm r_2)_{\mathcal{T}_h} -(\Pi_{W}{\phi},\nabla\cdot \bm r_2)_{\mathcal{T}_h} +\langle\Pi_{k+1}^{\partial}{\phi},\bm r_2 \cdot\bm n \rangle_{\partial\mathcal{T}_h}\\
	&\quad+(\nabla\cdot\bm{\Pi}_{V}\bm p,w_2)_{\mathcal{T}_h}-\langle  \bm{\Pi}_{V}\bm p\cdot \bm n, \mu_2\rangle_{\partial\mathcal{T}_h}\\
	&\quad+\langle \tau(\Pi_{W}{\phi}-\Pi_{k+1}^{\partial}{\phi}), w_2-\mu_2\rangle_{\partial\mathcal{T}_h}
	{-\langle \bm p\cdot
		\bm n,\mu_2\rangle_{\partial {\cal T}_h}.}
	\end{align*}
	By the definition  of $\bm{\Pi}_{V}$ and $\Pi_{W}$ in \eqref{HDG_projection_operator} we get 
	\begin{align*}
	\hspace{1em}&\hspace{-1em}  \mathscr B(\bm{\Pi}_{V}\bm p,\Pi_{W} \phi,\Pi_{k+1}^{\partial}{\phi}; \bm r_2,w_2,\mu_2)\\
	&=(\bm\Pi_V \bm p - \bm p,\bm r_2)_{\mathcal{T}_h} + (\bm p,\bm r_2)_{\mathcal{T}_h} -({\phi},\nabla\cdot \bm r_2)_{\mathcal{T}_h}+\langle {\phi},\bm r_2 \cdot\bm n \rangle_{\partial\mathcal{T}_h}\\
	&\quad+(\nabla\cdot(\bm{\Pi}_{V}\bm p - \bm p),w_2)_{\mathcal{T}_h} +(\nabla\cdot\bm p,w_2)_{\mathcal{T}_h} + \langle  (\bm p - \bm{\Pi}_{V}\bm p)\cdot \bm n, \mu_2\rangle_{\partial\mathcal{T}_h}\\
	&\quad+\langle \tau (\Pi_{W}{\phi}-\Pi_{k+1}^{\partial}{\phi}),  w_2-\mu_2\rangle_{\partial\mathcal{T}_h}\\
	&{=\mathscr B(\bm p,\phi,\phi; \bm r_2,w_2,\mu_2)+(\bm\Pi_V \bm p - \bm p,\bm r_2)_{\mathcal{T}_h} +(\nabla\cdot(\bm{\Pi}_{V}\bm p - \bm p),w_2)_{\mathcal{T}_h}} \\&{\quad  + \langle  (\bm p - \bm{\Pi}_{V}\bm p)\cdot \bm n, \mu_2\rangle_{\partial\mathcal{T}_h}+\langle \tau (\Pi_{W}{\phi}-\Pi_{k+1}^{\partial}{\phi}),  w_2-\mu_2\rangle_{\partial\mathcal{T}_h}}
	\end{align*}
	Since 
	\begin{align*}
	(\nabla\cdot(\bm{\Pi}_{V}\bm p - \bm p),w_2)_{\mathcal{T}_h}  &=  \langle (\bm{\Pi}_V \bm p - \bm p)\cdot \bm n, w_2  \rangle_{\partial \mathcal T_h}- (\bm{\Pi}_{V}\bm p - \bm p,\nabla w_2)_{\mathcal{T}_h} \\
	&=  \langle (\bm{\Pi}_{V}\bm p - \bm p)\cdot \bm n, w_2  \rangle_{\partial \mathcal T_h}.
	\end{align*}
	We have
	{\begin{align*}
		\hspace{1em}&\hspace{-1em}  \mathscr B(\bm{\Pi}_{V}\bm p,\Pi_{W} \phi,\Pi_{k+1}^{\partial}{\phi}; \bm r_2,w_2,\mu_2)\\
		&=\mathscr B(\bm p,\phi,\phi; \bm r_2,w_2,\mu_2)+(\bm\Pi_V \bm p - \bm p,\bm r_2)_{\mathcal{T}_h} \\&\quad+\langle  (\bm p - \bm{\Pi}_{V}\bm p)\cdot \bm n, \mu_2-w_2\rangle_{\partial\mathcal{T}_h}+\langle \tau (\Pi_{W}{\phi}-\Pi_{k+1}^{\partial}{\phi}),  w_2-\mu_2\rangle_{\partial\mathcal{T}_h}.
		\end{align*}}
	{Using  the analogue of \Cref{Drift_Diffusion_HDG_Formulation_Compact_second} for the exact solution,  and \eqref{HDG_projection_operator} we get} 	 
	\begin{align*}
	\mathscr B(\bm{\Pi}_{V}\bm p,\Pi_{W} \phi,\Pi_{k+1}^{\partial}{\phi}; \bm r_2,w_2,\mu_2) =(\bm\Pi_V \bm p - \bm p,\bm r_2)_{\mathcal{T}_h} - ( u, w_2)_{\mathcal T_h}.
	\end{align*}
	Therefore, subtracting \Cref{Drift_Diffusion_HDG_Formulation_Compact_second} we have the following error equation
	\begin{align*}
	\mathscr B(\xi_h^{\bm p},\xi_h^{{\phi}},\xi_h^{\widehat {\phi}};
	\bm r_2,w_2,\mu_2)=(\bm\Pi_V \bm p - \bm p,\bm r_2)_{\mathcal{T}_h} - ( u - u_h, w_2)_{\mathcal T_h}.
	\end{align*}
\end{proof}

\subsubsection{$L^2$ Error estimates for $p$ and $\phi$.}
\begin{lemma} \label{estimate_p1}
	We have the following estimate
	\begin{align*}
	\|\xi_h^{\bm p}\|^2_{\mathcal{T}_h}
	+\|\sqrt{\tau}(\Pi_{k+1}^{\partial}\xi_h^{{\phi}}-\xi_h^{\widehat {\phi}})\|^2_{\partial\mathcal{T}_h} \le \|u-u_h\|_{\mathcal{T}_h}\|\xi_h^{{\phi}}\|_{\mathcal{T}_h}.
	\end{align*}
\end{lemma}
\begin{proof} 
	We take $(\bm r_2,w_2,\mu_2)=(\xi_h^{\bm p},\xi_h^{{\phi}},\xi_h^{\widehat{{\phi}}} )$ in \eqref{error_equation2} to get
	\begin{align*}
	\mathscr B(\xi_h^{\bm p},\xi_h^{{\phi}},\xi_h^{\widehat {\phi}};
	\xi_h^{\bm p},\xi_h^{{\phi}},\xi_h^{\widehat {\phi}}) =-(u-u_h,\xi_h^{{\phi}})_{\mathcal{T}_h} \le \|u-u_h\|_{\mathcal{T}_h}\|\xi_h^{{\phi}}\|_{\mathcal{T}_h}.
	\end{align*}
	On the other hand, by \Cref{basis_property1} we have 
	\begin{align*}
	\|\xi_h^{\bm p}\|^2_{\mathcal{T}_h}
	+\|\sqrt{\tau}(\xi_h^{{\phi}}-\xi_h^{\widehat {\phi}})\|^2_{\partial\mathcal{T}_h} \le \|u-u_h\|_{\mathcal{T}_h}\|\xi_h^{{\phi}}\|_{\mathcal{T}_h}.
	\end{align*}
	%		By \Cref{energy_argument1,HDG_Poincare} to get
	%		\begin{align*}
	%			\hspace{1em}&\hspace{-1em}\|\xi_h^{\bm p}\|^2_{\mathcal{T}_h}
	%			+\|\sqrt{\tau}(\Pi_{k+1}^{\partial}\xi_h^{{\phi}}-\xi_h^{\widehat {\phi}})\|^2_{\partial\mathcal{T}_h}\\
	%			&\le  C\|u-u_h\|_{\mathcal{T}_h}
	%			\left(
	%			\|\xi_h^{\bm p}\|_{\mathcal{T}_h}
	%			+\|h_K^{-1/2}(\Pi_{k+1}^{\partial}\xi_h^\phi-\xi_h^{\widehat \phi})\|_{\partial\mathcal{T}_h}
	%			\right).
	%		\end{align*}
	%		The Cauchy-Schwarz inequality gives the desired result.
\end{proof}
If we directly apply \Cref{HDG_Poincare} to get the estimate of $\|\xi_h^{{\phi}}\|_{\mathcal{T}_h}$, we will obtain only suboptimal convergence rates. {To obtain optimal rates we use the
	dual problem introduced in equation (\ref{D2}) with $\bm p=0$ and $M=0$ and assume the regularity estimate (\ref{elliptic_regularity}).}

We follow the proof of \Cref{error_equation} to get the following lemma.
\begin{lemma}\label{error_equation_dual1} {Let $(\bm\Phi,\Psi)$ solve (\ref{D2}) with $\bm p=0$ and $M=0$ having data
		$\Theta$. Then for}
	any	$(\bm r_2, w_2,\mu_2)\in \bm S_h\times \Psi_h\times\widehat \Psi_{h}(0)$, we have the following  equation
	\begin{align*}
	\mathscr B(\bm{\Pi}_{V}\bm \Phi,\Pi_{W} \Psi,\Pi_{k+1}^{\partial}{\Psi}; \bm r_2,w_2,\mu_2) =(\bm\Pi_V \bm \Phi - \bm \Phi,\bm r_2)_{\mathcal{T}_h} + ( \Theta, w_2)_{\mathcal T_h}.
	\end{align*}
\end{lemma}
{Using this lemma we can now estimate $\xi_h^{\phi}$ in terms of $u-u_h$ and other consistency terms.}
\begin{lemma}\label{lemma:step3_first_lemma}
	For any $t\in [0,T]$, if the elliptic regularity inequality \eqref{elliptic_regularity} holds,  then we have the following error estimates
	\begin{align*}
	\|\xi_h^{{\phi}}\|^2_{\mathcal{T}_h} \le Ch^2 \|\bm \Pi_V \bm p - \bm p\|_{\mathcal T_h}^2 +C\|u -u_h \|_{\mathcal T_h}^2. 
	\end{align*}
\end{lemma}

\begin{proof}
	Consider the dual problem \eqref{D2} {with $\bm p=0$ and $M=0$} and $\Theta = \xi_h^{{\phi}}$. We take  $(\bm r_2,w_2,\mu_2) = (-\bm\Pi_V\bm{\Phi},\Pi_W\Psi,\Pi_{k+1}^\partial \Psi)$ in \Cref{error_equation2} of  \Cref{error_equation} to get{
		\begin{align}\label{dual_eq2}
		\mathscr B(\xi_h^{\bm p},\xi_h^{{\phi}},\xi_h^{\widehat {\phi}}
		;-\bm\Pi_V\bm{\Phi},\Pi_W\Psi,\Pi_{k+1}^\partial\Psi) = -(\bm{\Pi}_V\bm p - \bm p,\bm\Pi_V\bm{\Phi})_{\mathcal T_h}- ( u - u_h, \xi_h^{{\phi}})_{\mathcal T_h}.
		\end{align}
		On the other hand, by \Cref{basis_property1,error_equation_dual1}, we have	
		\begin{align}\label{dual_eq1}
		\begin{split}
		\mathscr B &(\xi_h^{\bm p},\xi_h^{{\phi}},\xi_h^{\widehat {\phi}}
		;-\bm\Pi_V\bm{\Phi},\Pi_W\Psi,\Pi_{k+1}^\partial\Psi)\\
		&=\mathscr B(\bm\Pi_V\bm{\Phi},\Pi_W\Psi,\Pi_{k+1}^\partial\Psi;-\xi_h^{\bm p},\xi_h^{{\phi}},\xi_h^{\widehat {\phi}})\\
		&=-(\bm \Pi_V \bm \Phi - \bm \Phi,\xi_h^{\bm p})_{\mathcal{T}_h} + \|\xi_h^{{\phi}}\|_{\mathcal{T}_h}^2.
		\end{split}
		\end{align}
		Comparing the above two equalities \eqref{dual_eq2} and \eqref{dual_eq1}} gives
	\begin{align*}
	\|\xi_h^{{\phi}}\|^2_{\mathcal{T}_h}
	&=(\bm \Pi_V \bm \Phi - \bm \Phi,\xi_h^{\bm p})_{\mathcal{T}_h}-(\bm{\Pi}_V\bm p - \bm p,\bm\Pi_V\bm{\Phi})_{\mathcal T_h}- ( u - u_h, \xi_h^{{\phi}})_{\mathcal T_h}\\
	&=(\bm \Pi_V \bm \Phi - \bm \Phi,\xi_h^{\bm p})_{\mathcal{T}_h}-(\bm{\Pi}_V\bm p - \bm p,\bm\Pi_V\bm{\Phi}- \bm \Phi)_{\mathcal T_h} \\
	&\quad -(\bm{\Pi}_V\bm p - \bm p,\bm \Phi)_{\mathcal T_h}- ( u - u_h, \xi_h^{{\phi}})_{\mathcal T_h}\\
	&=(\bm \Pi_V \bm \Phi - \bm \Phi,\xi_h^{\bm p})_{\mathcal{T}_h}-(\bm{\Pi}_V\bm p - \bm p,\bm\Pi_V\bm{\Phi}- \bm \Phi)_{\mathcal T_h} \\
	&\quad +(\bm{\Pi}_V\bm p - \bm p,\nabla\Psi)_{\mathcal T_h}- ( u - u_h, \xi_h^{{\phi}})_{\mathcal T_h}\\
	&=(\bm \Pi_V \bm \Phi - \bm \Phi,\xi_h^{\bm p})_{\mathcal{T}_h}-(\bm{\Pi}_V\bm p - \bm p,\bm\Pi_V\bm{\Phi}- \bm \Phi)_{\mathcal T_h} \\
	&\quad +(\bm{\Pi}_V\bm p - \bm p,\nabla(\Psi-\Pi_W\Psi))_{\mathcal T_h}- ( u - u_h, \xi_h^{{\phi}})_{\mathcal T_h}\\
	&\le C h^2 \|\xi_h^{{\bm p}}\|_{\mathcal T_h}^2 + Ch^2 \|\bm \Pi_V \bm p - \bm p\|_{\mathcal T_h}^2 +C\|u -u_h \|_{\mathcal T_h}^2 + \frac 1 2  \|\xi_h^{{\phi}}\|^2_{\mathcal{T}_h}.
	\end{align*}
	By \Cref{estimate_p1} and the Cauchy-Schwarz inequality we obtain the result of the lemma:
	\begin{align*}
	\|\xi_h^{{\phi}}\|^2_{\mathcal{T}_h} \le Ch^2 \|\bm \Pi_V \bm p - \bm p\|_{\mathcal T_h}^2 +C\|u -u_h \|_{\mathcal T_h}^2. 
	\end{align*}
\end{proof}

As a consequence of the above result, a simple application of the triangle inequality and \Cref{estimate_p1,lemma:step3_first_lemma} give the  following bounds of $\|\phi -\phi_h\|_{\mathcal T_h}$ and $\|\bm p - \bm p_h\|_{\mathcal T_h}$:
\begin{lemma}\label{error_estimate_p_phi}
	Let $( \bm p, \phi)$ and $(\bm p_h, \phi_h)$ be the solution of \eqref{Drift_Diffusion_Main_Equation_Mixed_Weak_Form} and \eqref{Drift_Diffusion_HDG_Formulation}, respectively. 	For any $t\in [0,T]$, if the elliptic regularity inequality \eqref{elliptic_regularity} holds, then we have the following error estimates
	\begin{align*}
	\| \phi- \phi_h\|_{\mathcal T_h}  + \| \bm p- \bm p_h\|_{\mathcal T_h}  \le  C_1  h^{k+2} + C \|u - u_h\|_{\mathcal T_h}
	\end{align*}
	{where $C_1$ depends on the $H^{k+1}(\Omega)$ norm of $\bm p$ at each time}.
\end{lemma}

\subsection{$L^2$ Error estimates for $u$.}
{Having the result of \Cref{error_estimate_p_phi} it remains to estimate $u-u_h$.  The fundamental estimate
	is contained in the next lemma.}
\begin{lemma}\label{main_result_on_subinterval}
	If $h$ small enough, then there exists $t_h^\star \in [0,T]$ such that for all $t\in [0, t_h^\star]$ we have
	\begin{align*} 
	\|\xi_h^u\|^2_{\mathcal{T}_h}+\int_{0}^t \left(\|\xi_h^{\bm q}\|^2_{\mathcal{T}_h} +\|h_K^{-1/2}(\Pi_k^{\partial}\xi_h^u-\xi_h^{\widehat u})\|^2_{\partial\mathcal{T}_h}\right)dt  \le Ch^{2k+4}.
	\end{align*}
\end{lemma}

\begin{proof} 
	We take $(\bm r_1,w_1,\mu_1)=(\xi_h^{\bm q},\xi_h^u,\xi_h^{\widehat u})$ in \eqref{error_equation1} to get
	\begin{align}\label{energy_pro2}
	\begin{split}
	\hspace{1em}&\hspace{-1em} (\partial_t\xi_h^u,\xi_h^u)_{\mathcal{T}_h}+
	\|\xi_h^{\bm q}\|^2_{\mathcal{T}_h}+\|h_K^{-1/2}(\Pi_k^{\partial}\xi_h^u-\xi_h^{\widehat u})\|^2_{\partial\mathcal{T}_h}\\
	&=(\partial_t(u_{Ih}-u),\xi^u_h)_{\mathcal{T}_h}+M(u-u_{Ih},\xi^u_h)_{\mathcal{T}_h}\\
	&\quad -( (\bm p-\bm p_h) u_h, \nabla \xi^u_h)_{\mathcal{T}_h} + \langle (\bm p - \widehat{\bm p}_h)\cdot \bm n \widehat u_h, \xi^u_h - \xi_h^{\widehat u} \rangle_{\partial\mathcal{T}_h}\\
	&\quad -( \bm p \xi^u_h, \nabla \xi^u_h)_{\mathcal{T}_h} + \langle \bm p \cdot \bm n   \xi_h^{\widehat u}, \xi^u_h \rangle_{\partial\mathcal{T}_h}\\
	&=:R_1+R_2+R_3+R_{4} + R_{5} +R_{6}.
	\end{split}
	\end{align}
	We note that $\xi_h^u(0) = u_h(0) - u_{Ih}(0) = 0$. Let $t=0$ in \eqref{energy_pro2} to get 
	\begin{align*}
	\|\xi_h^{\bm q}(0)\|^2_{\mathcal{T}_h}+\|h_K^{-1/2}(\Pi_k^{\partial}\xi_h^u(0)-\xi_h^{\widehat u}(0))\|^2_{\partial\mathcal{T}_h} = 0.
	\end{align*}
	This implies $\xi_h^{\widehat u}(0) = \xi_h^u(0)=0$. Hence we have $\widehat u_h(0) = \widehat u_{Ih}(0)$. By \Cref{HDG_elliptic_projection_approximation} {we have} 
	\begin{align*}
	\| \Pi_{k+1}^o u(0) -u_h(0)\|_{\mathcal T_h} &=  \|\Pi_{k+1}^o u(0) -  u_{Ih}(0)\|_{\mathcal T_h} \le Ch^{k+2},\\
	\|\Pi_k^\partial u(0) - \widehat u_h(0)\|_{\partial \mathcal T_h}& = \|\Pi_k^\partial u(0) - \widehat u_{Ih}(0)\|_{\partial \mathcal T_h} \le Ch^{k+3/2}.
	\end{align*}
	For $h$ small enough {these estimates imply that} 
	\begin{align}
	\|u(t) - \Pi_{k+1}^o u(t)\|_{L^\infty(\Omega)}\le 1/2 \textup{ and } 	\|u(t) - \Pi_{k}^\partial  u(t)\|_{L^\infty(\mathcal E_h)} \le 1/2 \textup{ for all  } t\in [0,T]. 
	\end{align}
	Let $\mathcal M = \max_{(t,x)\in [0,T]\times \Omega} |u(t,x)|$, then the inverse inequality gives
	\begin{align*}
	\|  u_h(0)\|_{L^\infty(\Omega)} &\le Ch^{-d/2}\| \Pi_{k+1}^o u(0) -u_h(0)\|_{\mathcal T_h}  \\
	&\quad + \| \Pi_{k+1}^o u(0) - u(0)\|_{L^\infty (\Omega)} +  \|u(0)\|_{L^\infty (\Omega)}\\
	&\le Ch^{k+2-d/2} + \mathcal M + 1/2,\\
	\| \widehat u_h(0)\|_{L^\infty(\mathcal E_h)} &\le Ch^{{1/2-d}/2}\| \Pi_k^\partial u(0) -\widehat u_h(0)\|_{\mathcal T_h}  \\
	&\quad + \| \Pi_k^\partial  u(0) - u(0)\|_{L^\infty (\mathcal E_h)} +  \|  u(0)\|_{L^\infty (\mathcal E_h)}\\
	&\le Ch^{k+2-d/2} + \mathcal M + 1/2.
	\end{align*}
	Also, since the error equation \eqref{error_equation1} is continuous with respect to the time $t$, then  there exists $t_h^\star\in [0,T]$ such that for $h$ small enough,
	\begin{align}\label{bounded}
	\|u_h\|_{L^\infty(\Omega)} +	\|\widehat u_h\|_{L^\infty(\mathcal E_h)} \le 2 \mathcal M+2. 
	\end{align}
	%		For the  terms $R_1 + R_2$, since $(\partial_t\xi_h^{\bm q},\partial_t\xi_h^{u},\partial_t\xi_h^{\widehat u})$ satisfies \eqref{time_error_a_derivative}, then by \Cref{energy_argument1} to get 
	%		\begin{align}\label{energy_est2}
	%		\|\nabla \partial_t\xi_h^{u}\|_{\mathcal T_h}^2 + \|h_K^{-1/2}( \partial_t\xi_h^{u} - \partial_t\xi_h^{\widehat u})\|_{\partial \mathcal T_h}^2 \le C 	\| \partial_t\xi_h^{\bm q}\|_{\mathcal T_h}^2 +C \|h_K^{-1/2}(\Pi_k^\partial \partial_t\xi_h^{u} - \partial_t\xi_h^{\widehat u})\|_{\partial \mathcal T_h}^2.
	%		\end{align}
	By the  Cauchy-Schwarz inequality,  \Cref{HDG_elliptic_projection_approximation,energy_argument1} we get
	\begin{align*}
	R_1 +R_2 &\le Ch^{k+2}\|\xi_h^u\|_{\mathcal{T}_h}\\
	&\le Ch^{2k+4}+\frac{1}{8}\left(\|\xi_h^{\bm q}\|^2_{\mathcal{T}_h}
	+\|h_K^{-1/2}(\Pi_k^{\partial}\xi_h^u-\xi_h^{\widehat u})\|^2_{\partial\mathcal{T}_h}
	\right).
	\end{align*}
	%A similar argument appy to the terms $R_7+R_8$ to get 
	%		\begin{align*}
	%			R_7 +R_8 \le Ch^{2k+4}+\frac{1}{8}\left(\|\xi_h^{\bm q}\|^2_{\mathcal{T}_h}+\|h_K^{-1/2}(\Pi_k^{\partial}\xi_h^u-\xi_h^{\widehat u})\|^2_{\partial\mathcal{T}_h}\right).
	%		\end{align*}
	For the term $R_3$, by the Cauchy-Schwarz,  \Cref{error_estimate_p_phi},  \Cref{HDG_Poincare} and \Cref{energy_argument1} we get 
	\begin{align*}
	R_3&\le C\|\bm p-\bm p_h\|_{\mathcal{T}_h}\|\nabla   \xi^u_h\|_{\mathcal{T}_h}\\
	&\le C\|\bm p-\bm p_h\|_{\mathcal{T}_h}^2 + \frac 1 C\|\nabla   \xi^u_h\|_{\mathcal{T}_h}^2\\
	&\le Ch^{2k+4} + C\| u - u_h\|_{\mathcal{T}_h}^2 + \frac 1 C\|\nabla   \xi^u_h\|_{\mathcal{T}_h}^2\\
	&\le Ch^{2k+4} +C \|\xi_h^u\|_{\mathcal T_h}^2+\frac{1}{8}\left(\|\xi_h^{\bm q}\|^2_{\mathcal{T}_h}+\|h_K^{-1/2}(\Pi_k^{\partial}\xi_h^u-\xi_h^{\widehat u})\|^2_{\partial\mathcal{T}_h}\right).
	\end{align*}
	%		By the same argument we have 
	%		\begin{align*}
	%			R_9 &\le Ch^{2k+4} +C \|\xi_h^u\|_{\mathcal T_h}^2+\frac{1}{8}\left(\|\xi_h^{\bm q}\|^2_{\mathcal{T}_h}
	%			+\|h_K^{-1/2}(\Pi_k^{\partial}\xi_h^u-\xi_h^{\widehat u})\|^2_{\partial\mathcal{T}_h}
	%			\right).
	%		\end{align*}
	Also, applying \Cref{energy_argument1} again to obtain
	\begin{align*}
	R_4&=\langle (\bm p-\widehat{\bm p}_h) \cdot\bm n \widehat u_h, \xi^u_h-\xi^{\widehat u}_h \rangle_{\partial\mathcal{T}_h}\\
	&\le C \|h_K^{1/2}(\bm p-\widehat{\bm p}_h)\|_{\partial\mathcal{T}_h}
	\|h_K^{-1/2}(\xi_h^u- \xi^{\widehat u}_h)\|_{\partial\mathcal{T}_h}\\
	&\le Ch^{2k+4}+C \|\xi_h^u\|_{\mathcal T_h}^2+\frac{1}{8}\left(\|\xi_h^{\bm q}\|^2_{\mathcal{T}_h}+\|h_K^{-1/2}(\Pi_k^{\partial}\xi_h^u-\xi_h^{\widehat u})\|^2_{\partial\mathcal{T}_h}\right).
	\end{align*}
	For the last two terms $R_5+R_6$,  integration by parts to get 
	\begin{align*}
	R_5 + R_6 &= -( \bm p \xi^u_h, \nabla \xi^u_h)_{\mathcal{T}_h} + \langle \bm p \cdot \bm n   \xi_h^{\widehat u}, \xi^u_h \rangle_{\partial\mathcal{T}_h}\\
	& = -\frac 1 2 \langle \bm p\cdot \bm n (\xi^u_h - \xi_h^{\widehat u}), \xi^u_h - \xi_h^{\widehat u} \rangle_{\mathcal T_h} -( \nabla\cdot \bm p \xi^u_h,  \xi^u_h)_{\mathcal{T}_h} \\
	&\le \frac 1 8 \|h_K^{-1/2}(\Pi_k^{\partial}\xi_h^u-\xi_h^{\widehat u})\|^2_{\partial\mathcal{T}_h} + \|\nabla\cdot\bm p\|_{L^\infty(\Omega)} \| \xi^u_h\|_{\mathcal T_h}^2.
	\end{align*}
	Sum the above estimates of $\{R_i\}_{i=1}^6$ to get 
	\begin{align}\label{important_result}
	(\partial_t\xi_h^u,\xi_h^u)_{\mathcal{T}_h}+
	\|\xi_h^{\bm q}\|^2_{\mathcal{T}_h}+\|h_K^{-1/2}(\Pi_k^{\partial}\xi_h^u-\xi_h^{\widehat u})\|^2_{\partial\mathcal{T}_h} \le C h^{2k+4} + C \| \xi^u_h\|_{\mathcal T_h}^2.
	\end{align}
	Integrating both sides of  \eqref{important_result} on $[0, t_h^*]$ we finally obtain
	\begin{align*}
	\hspace{3em}&\hspace{-3em} \|\xi_h^u(t_h^*)\|^2_{\mathcal{T}_h}+\int_{0}^{t_h^*} \left(\|\xi_h^{\bm q}\|^2_{\mathcal{T}_h} +\|h_K^{-1/2}(\Pi_k^{\partial}\xi_h^u-\xi_h^{\widehat u})\|^2_{\partial\mathcal{T}_h}\right) dt\\
	&\le Ch^{2k+4} + C \int_0^{t_h^*} \|\xi_h^u\|_{\mathcal T_h}^2 dt.
	\end{align*}
	The use of Gronwall's inequality gives the desired result.
\end{proof}

\begin{lemma}\label{theorem_err_u3_extend} 
	For $h$ small enough, the result in  \Cref{main_result_on_subinterval}  {holds} on the whole time interval $[0,T]$.
\end{lemma}
\begin{proof}
	Fix $ h^* > 0 $ so that \Cref{main_result_on_subinterval} is  true for all $ h \leq h^* $, and assume $t^*_h$ is the largest value for which \eqref{bounded}  is true for all $ h \leq h^* $.  Define the set $ \mathbb {A} = \{ h \in [0,h^*] : t^*_h \neq T \} $.  If the result is not true, then $ \mathbb {A} $ is nonempty, $ \inf \{ h : h\in \mathbb {A} \} = 0 $, and also
	\begin{align}\label{assum}
	\|u_h\|_{L^\infty(\Omega)}+ \|\widehat u_h\|_{L^\infty(\mathcal E_h)} = 2\mathcal M+2 \quad  \mbox{for all $ h \in \mathcal{A} $.}
	\end{align}
	However, by the inverse inequality and since \Cref{main_result_on_subinterval} {holds}, we have
	\begin{align*}
	\|u_h\|_{L^\infty(\Omega)}+ \|\widehat u_h\|_{L^\infty(\mathcal E_h)} \le  C h^{2-d/2} + 2\mathcal M +1 \quad  \mbox{for all $ h \in \mathcal{A} $.}
	\end{align*}
	Since $ C $ does not depend on $ h $, there exists $ h^*_1 \leq h^* $ such that $ 		\|u_h\|_{L^\infty(\Omega)}+ \|\widehat u_h\|_{L^\infty(\mathcal E_h)} <2\mathcal M +2 $ for all $ h \in \mathbb{A} $ such that $ h \leq h^*_1 $.  This contradicts \eqref{assum}, and therefore $t^*_h = T$ for all $ h $ small enough.
\end{proof}

The above lemma, the triangle inequality, and \Cref{estimate_p1} complete the proof of \Cref{main_theorem}.

\section{Numerical Results}
\label{Numerical_Results}
In this section we present some numerical results in two spatial dimensions.
\begin{example}\label{ex1}
	We begin with an example with an exact solution in order to illustrate the convergence theory.  The domain is the unit square $\Omega = [0,1]\times [0,1]\subset \mathbb R^2$ and {homogeneous} Dirichlet boundary conditions are applied on the boundary. The source terms $f_1$, $f_2$ and the initial condition are chosen so that $\varepsilon = 0.1$ and  the exact solution $u = \cos(t)\sin(x)\cos(y)$ and $\phi = \sin(t)\cos(x)\sin(y)$.  The second order backward differentiation formula (BDF2) is  applied for the time discretization  and for the space discretization we choose polynomial degrees $ k = 0 $ or $ k = 1 $ {(used in the definition of the discrete spaces in Section~\ref{intro}).}.  The time step is chosen to be $\Delta t = h$ when $k=0$ and $\Delta t   = h^{3/2}$ when $k=1$.  We report the errors at the final time $ T = 1 $.  The observed convergence rates match our theory.
	\begin{table}%[H]
		\small
		\caption{{History of convergence for $\bm{q}_h$ and $u_h$  for Example \ref{ex1} under uniform mesh refinement.}}\label{table_1}
		\centering
		%%%%%%%%%%%%%%%%%%%%%%%%%%%%%%%55
		\begin{tabular}{c|c|c|c|c|c}
			\Xhline{1pt}
			\multirow{2}{*}{Degree}
			&\multirow{2}{*}{$\frac{h}{\sqrt{2}}$}	
			&\multicolumn{2}{c|}{$\|\bm{q}-\bm{q}_h\|_{0,\Omega}$}	
			&\multicolumn{2}{c}	{$\|\bm p-\bm p_h\|_{0,\Omega}$}		\\
			\cline{3-6}
			& &Error &Rate
			&Error &Rate
			\\
			
			\cline{1-6}
			\multirow{5}{*}{ $k=0$}
			&$2^{-1}$	&4.2730E-02 	&	    &6.4843E-03&	   \\
			&$2^{-2}$	&2.2386E-02	&0.93	& 1.9113E-03 	&1.76\\
			&$2^{-3}$	& 1.1265E-02	&0.99	&5.1822E-04 	&1.88 \\
			&$2^{-4}$	&5.6455E-03	&1.00	&1.3592E-04	&1.93 \\
			&$2^{-5}$	&2.8248E-03&1.00 	&3.4881E-05	&1.96 \\

			\cline{1-6}
			\multirow{5}{*}{ $k=1$}
			&$2^{-1}$	&2.9547E-03	&	    & 3.8888E-04 &	   \\
			&$2^{-2}$	&7.5335E-04&1.97	& 5.4882E-05	&2.82\\
			&$2^{-3}$	& 1.9796E-04	&1.93	&7.5341E-06	&2.86 	\\
			&$2^{-4}$	&5.0451E-05 &1.97	& 9.8858E-07	&2.93 \\
			&$2^{-5}$	&1.2748E-05&1.98 	&1.2705E-07&2.96 \\

			\Xhline{1pt}

		\end{tabular}
		
	\end{table}

	\begin{table}%[H]
		\small
		\caption{{History of convergence for  $\bm p_h$ and  $\phi_h$ for Example \ref{ex1} under uniform mesh refinement.}}\label{table_2}
		\centering
		%%%%%%%%%%%%%%%%%%%%%%%%%%%%%%%55
		\begin{tabular}{c|c|c|c|c|c}
			\Xhline{1pt}
			\multirow{2}{*}{Degree}
			&\multirow{2}{*}{$\frac{h}{\sqrt{2}}$}		
			&\multicolumn{2}{c|}{$\|u-u_h\|_{0,\Omega}$}
			&\multicolumn{2}{c}{$\|\phi-\phi_h\|_{0,\Omega}$}	\\
			\cline{3-6}
			& &Error &Rate
			&Error &Rate
			
			\\
			
			\cline{1-6}
			\multirow{5}{*}{ $k=0$}
			&$2^{-1}$	    &1.8339E-02 & & 1.0205E-02&\\
			&$2^{-2}$		&4.9503E-03	&1.89&2.3408E-03 	&2.12\\
			&$2^{-3}$	&1.2423E-03 	&2.00 &4.9774E-04 	&2.23\\
			&$2^{-4}$	&3.1156E-04 	&2.00&1.1131E-04 	&2.16\\
			&$2^{-5}$	 	& 7.7965E-05	&2.00&2.6001E-05	&2.09\\

			\cline{1-6}
			\multirow{5}{*}{ $k=1$}
			&$2^{-1}$    &1.8339E-02 & & 4.0894e-04&\\
			&$2^{-2}$	&2.3140E-04	&2.98&3.7700E-05	&3.43\\
			&$2^{-3}$	& 2.9565E-05 	&2.97 &4.6167E-06	&3.03\\
			&$2^{-4}$	& 3.7026E-06	&3.00&5.3872E-07	&3.01\\
			&$2^{-5}$	& 4.6363E-07	&3.00&6.6418E-08	&3.02\\

			\Xhline{1pt}

		\end{tabular}
		
	\end{table}
\end{example}

Next, we test an example without a convergence rate but that show the performance of the HDG method.  We take $k=0$, the domain is also the unit square $\Omega = [0,1]\times [0,1]\subset \mathbb R^2$ and partition into $20000$ triangles, i.e., $h=\sqrt{2}/100$. BDF2 is applied for time discretization and the time step $\Delta t = 1/1000$.

\begin{example} \label{ex2}{This example has non-homogeneous Dirichlet data and demonstrates that our HDG scheme 
		can handle this case.} We take $\varepsilon =10^{-2}$ and the source terms $f_1=0$ and 
	\begin{align*}
	f_2 =
	\begin{cases}
	-0.8&  \text{$(0,0.5)\times (1/2,1)$},\\
	0.8& \text{else}.
	\end{cases}
	\end{align*}
	The Dirichlet boundary condition $g_{u}=0.9, g_{\phi}=1.1$ on $\{y=0\}$, and  $g_{u}=0.1, g_{\phi}=-1.1$ on $\{y=1, 0\le x\le 0.25\}$. Elsewhere we impose homogeneous Neumann boundary conditions. Initial condition $u_0 = (1+f_2)/2$.  A similar example was studied in \cite{Bessemoulin_FVM_2014_SINUM} by a finite volume method.  We plot the solutions $u_h$ and $\phi_h$ at different final time $T$; see \Cref{uh,phih2}.
	\begin{figure}
		\centerline{
			\hbox{\includegraphics[height=2in]{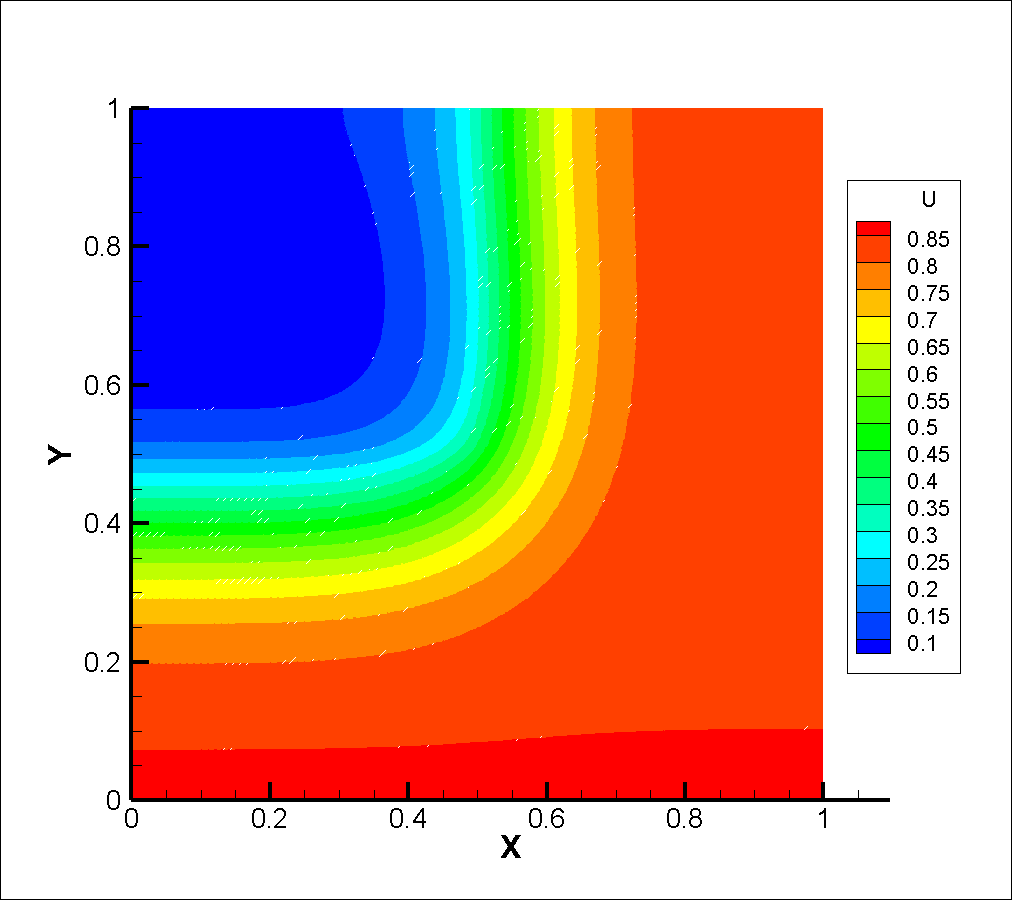}}
			\hbox{\includegraphics[height=2in]{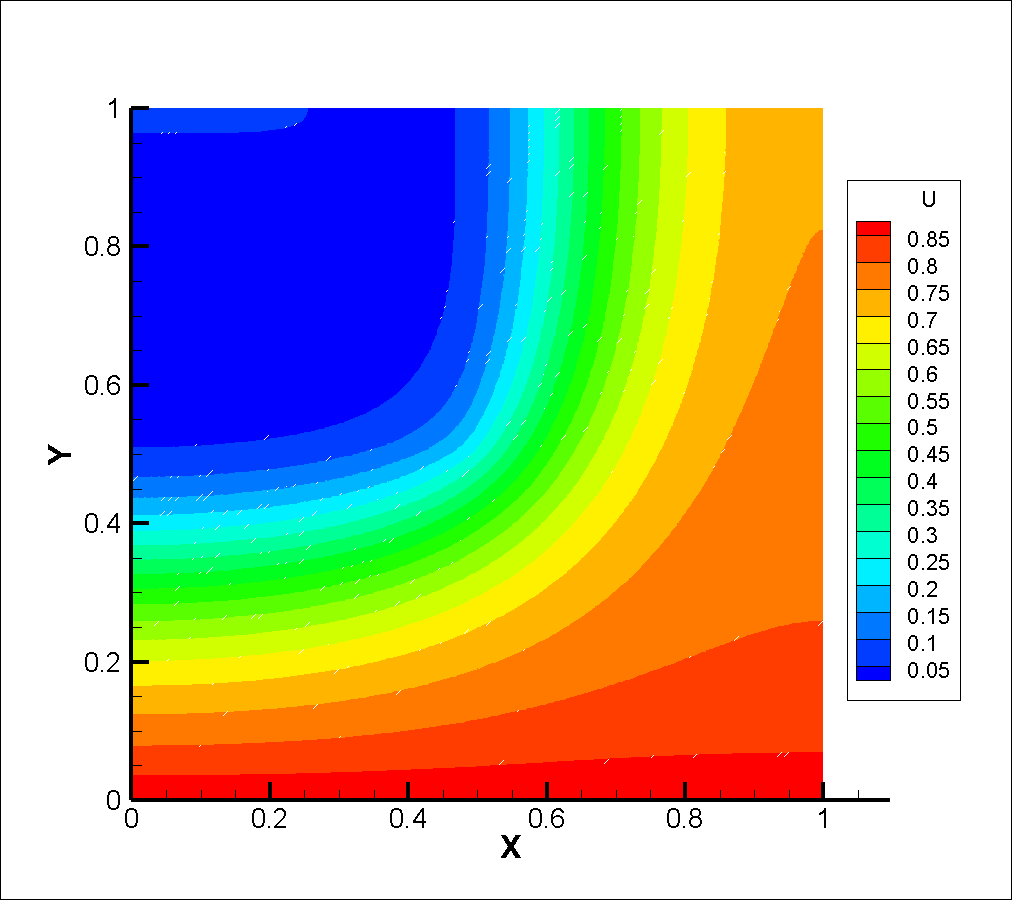}}}
		\centerline{
			\hbox{\includegraphics[height=2in]{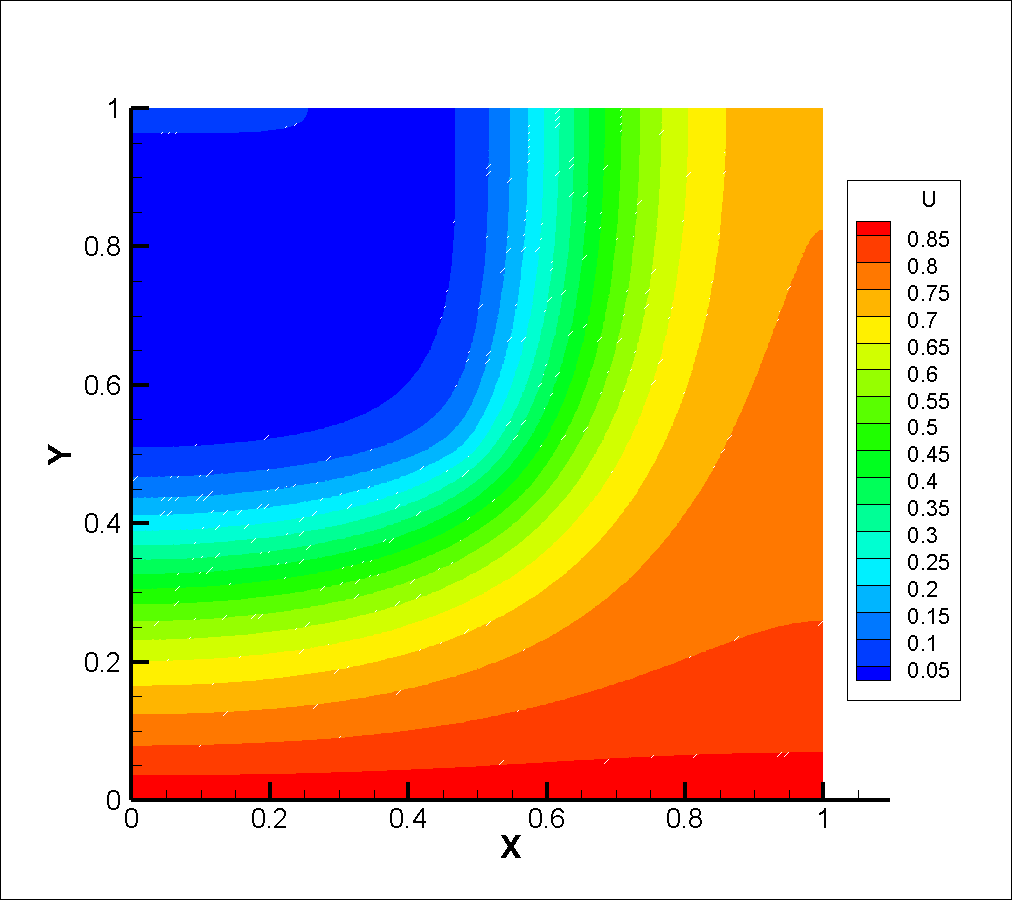}}
			\hbox{\includegraphics[height=2in]{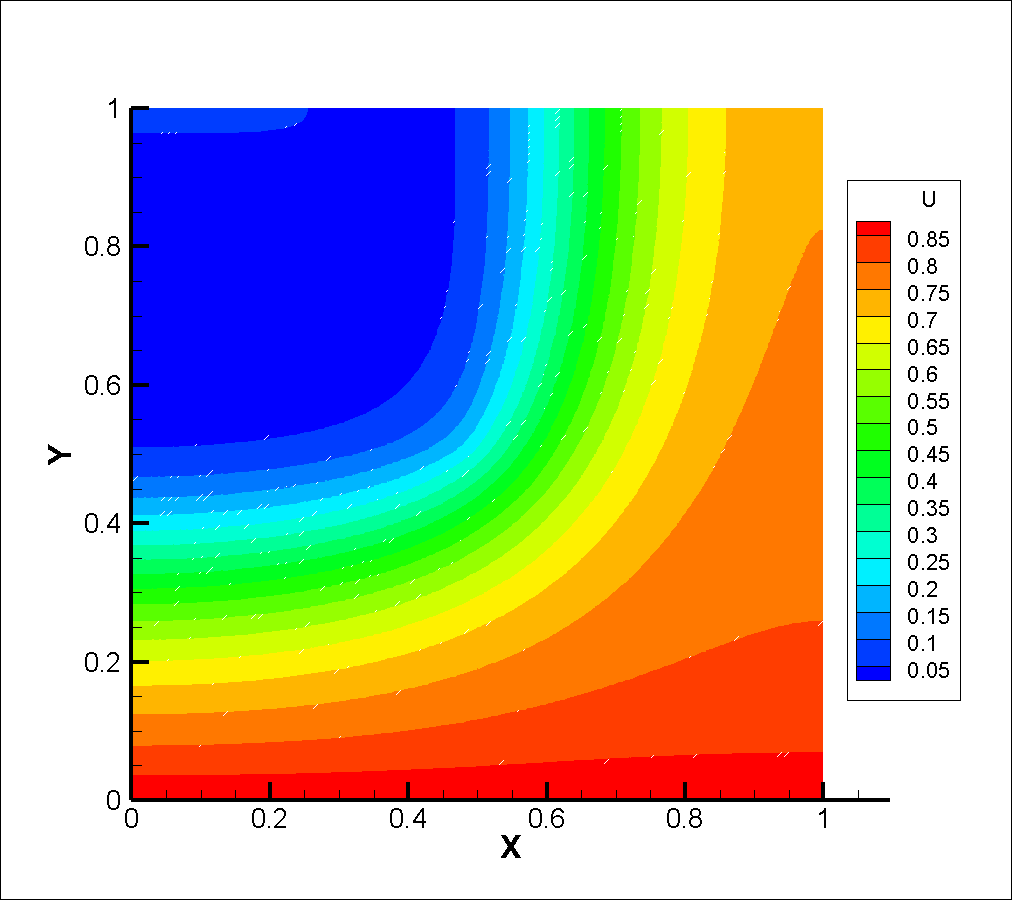}}}
		\caption{From left to right, from top to bottom are the contour plots of $u_h$ at  time: $T=0.01, 0.4,0.7,1$ for Example~\ref{ex2}.}
		\label{uh}
		\centering
	\end{figure}
	\begin{figure}
		\centerline{
			\hbox{\includegraphics[height=2in]{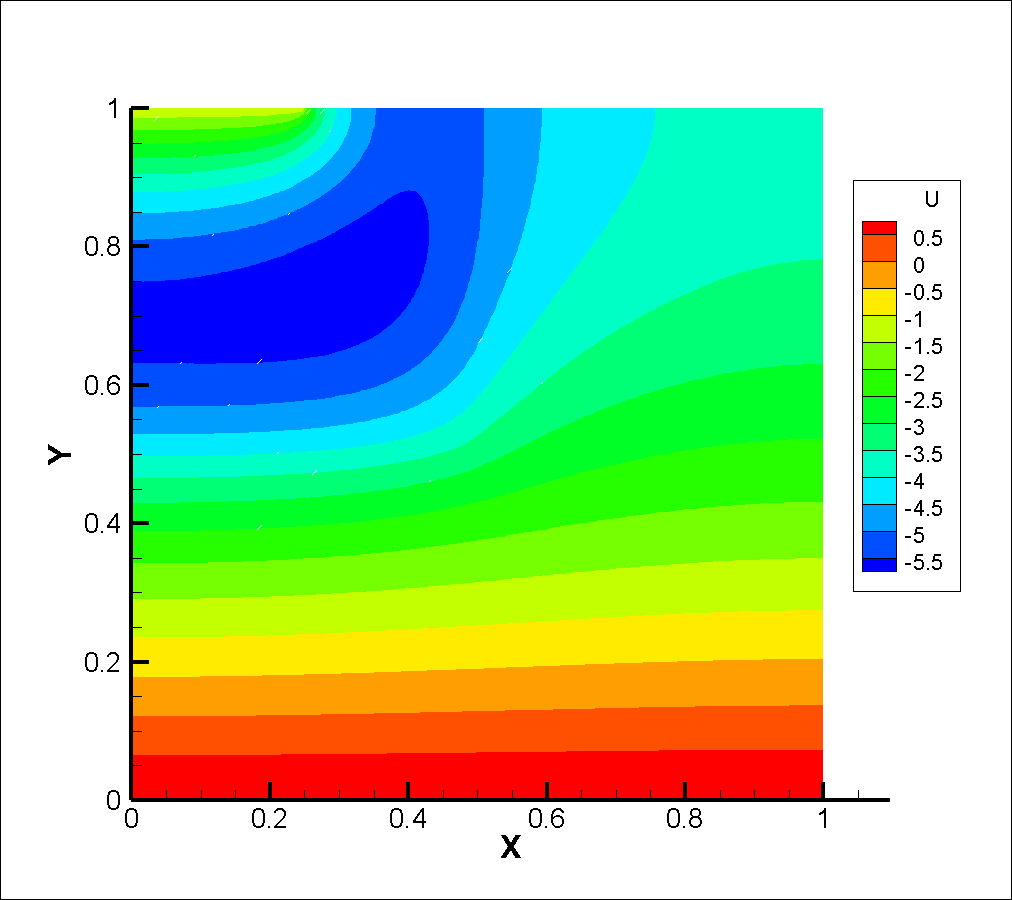}}
			\hbox{\includegraphics[height=2in]{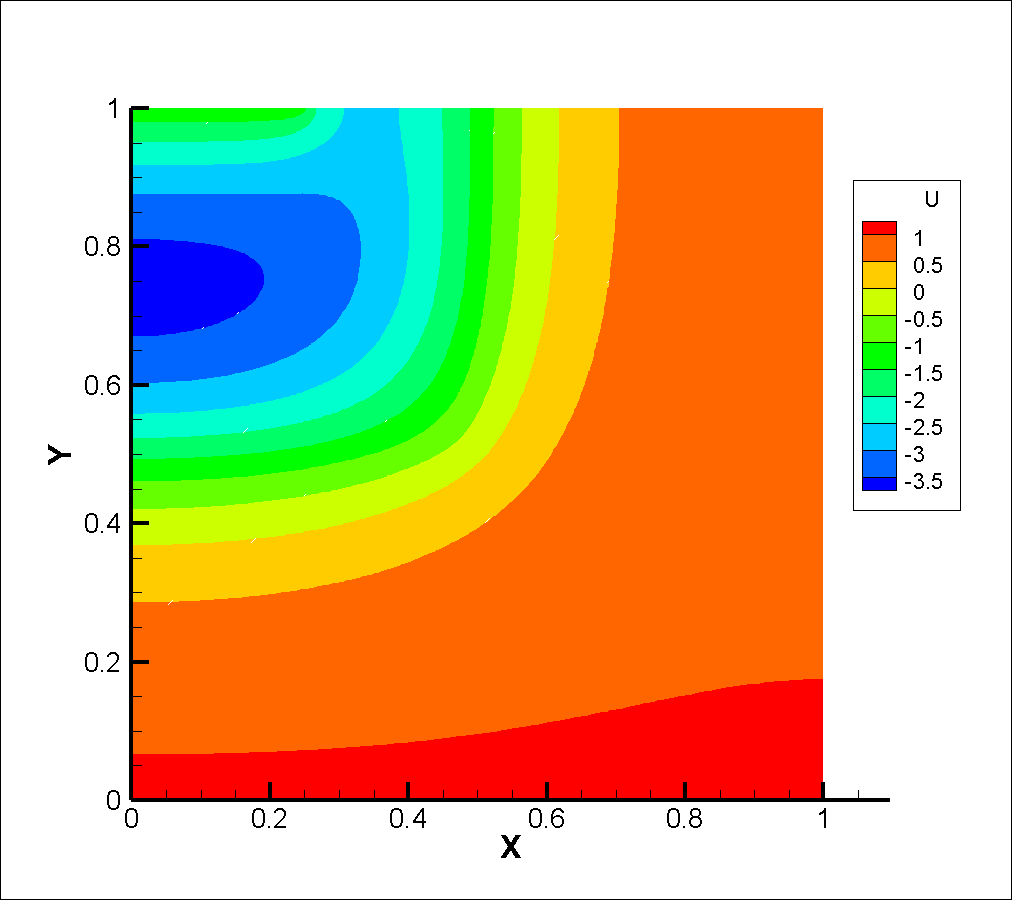}}}
		\centerline{
			\hbox{\includegraphics[height=2in]{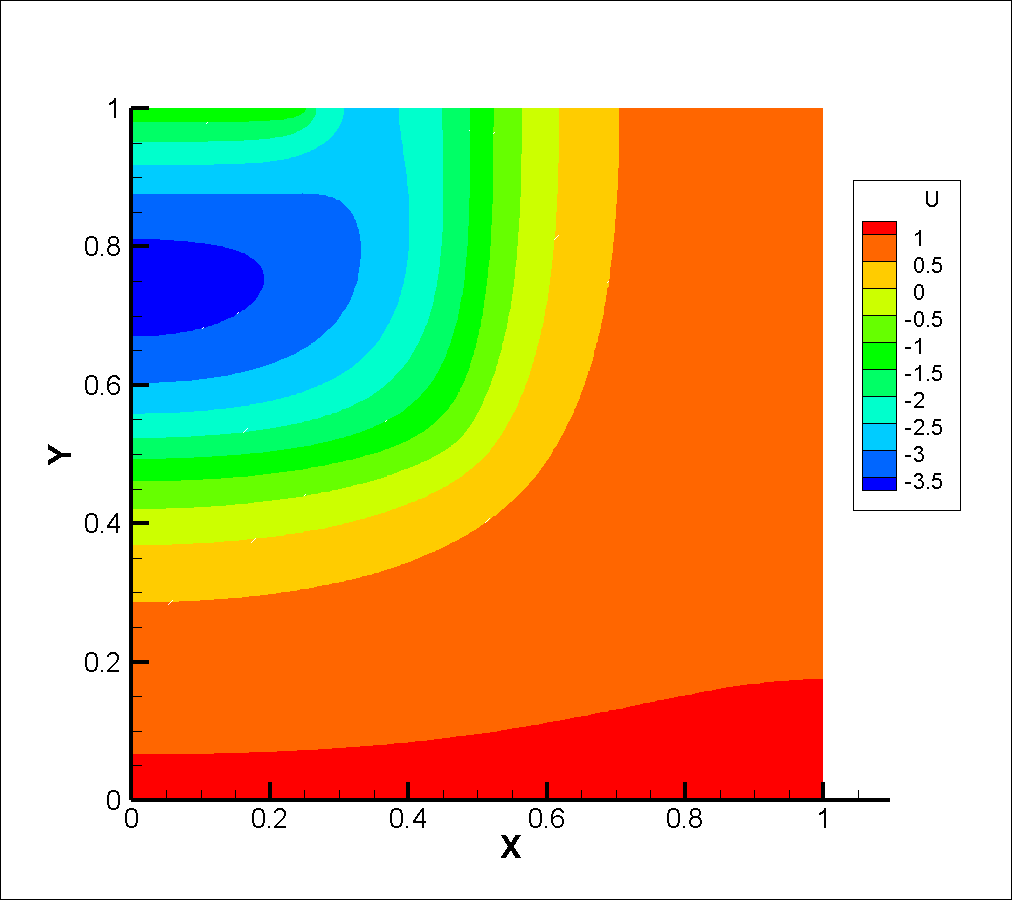}}
			\hbox{\includegraphics[height=2in]{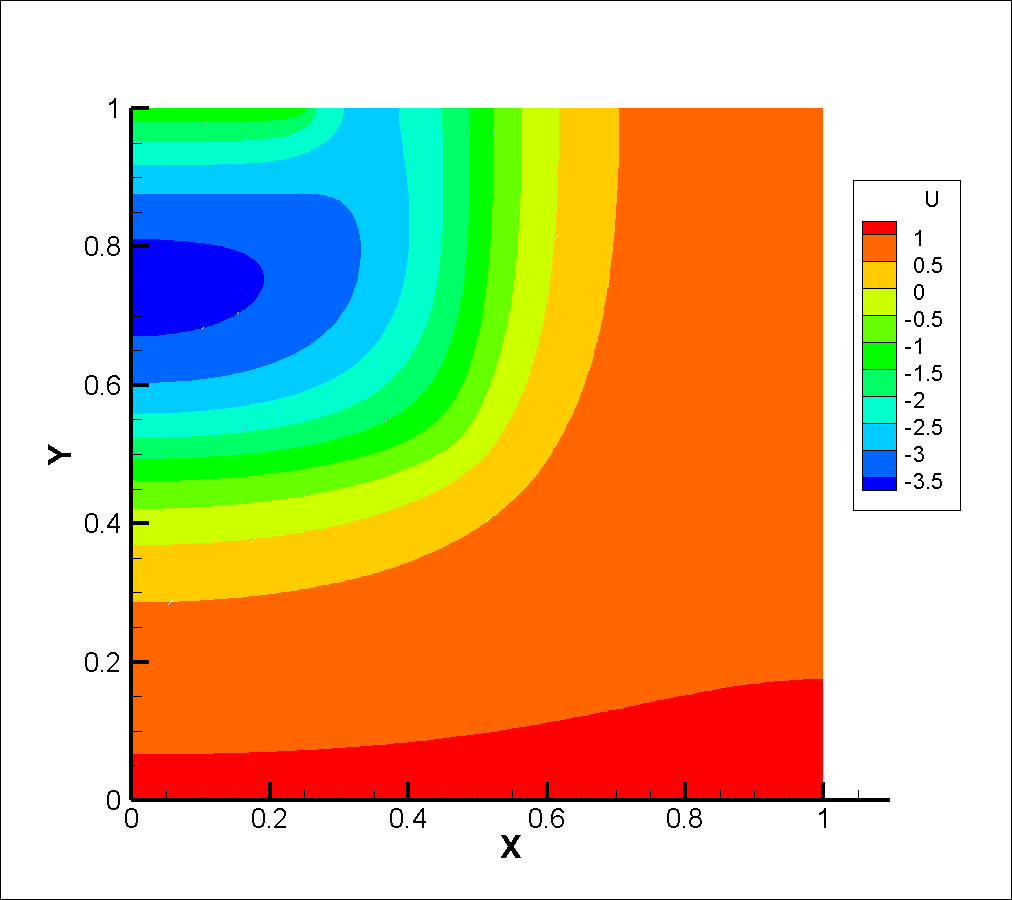}}}
		\caption{From left to right, from top to bottom are the contour plots of  $\phi_h$ at  time: $T=0.01, 0.4,0.7,1$ for Example~\ref{ex2}.}
		\label{phih2}
		\centering
	\end{figure}
\end{example}

\section{Conclusion}
In this work, we proposed an HDG method for the drift-diffusion equation.  We proved  optimal 
{semi-discrete} error estimates  for all variables;  moreover, from the point view of degrees of freedom, we obtained a superconvergent convergence rate for the variable $u$. {As far as we are aware, this is the first such result in the literature.}

{Clearly it would be desirable to prove convergence without the need to assume an inverse assumption.  Equally, it would be useful to prove fully discrete estimates using, for example BDF2 in time. }

This is the first of a series of papers in which we develop efficient HDG methods for drift-diffusion equation, including devising HDG methods when $\varepsilon$ approaches to zero. We have a great interest in the numerical solution of steady state drift-diffusion equation, and we will explore this problem in our future papers.

\section*{Acknowledgements}
G.\ Chen is supported by National natural science Foundation of China (NSFC) under grant number 11801063 and China Postdoctoral Science Foundation under grant number 2018M633339.
The research of  P.B.  Monk and Y. Zhang is partially supported by  the US
National Science Foundation (NSF) under grant number DMS-1619904.
\appendix
\section{Appendix}
\label{Appendix}

In this section we  give a proof for {
	\eqref{HDG_elliptic_projection_approximation1} and \eqref{HDG_elliptic_projection_approximation2}.  The proof of (\ref{HDG_elliptic_projection_approximation3}) is similar and we do not provide details. }

\subsection{Error equations}
{We start be deriving equations satisfied by standard projections (see (\ref{L2_projection})) of the exact solution.}
\begin{lemma}\label{projetion_error}
	Let $(\bm q, u)$ be {components of} the solution of \eqref{Drift_Diffusion_Main_Equation_Mixed_Weak_Form}, then we have 
	\begin{align*}
	\hspace{0.1em}&\hspace{-0.1em} M(\Pi_{k+1}^ou, w_1)_{\mathcal{T}_h}
	+\mathscr A(\bm{\Pi}_k^o\bm q, \Pi_{k+1}^ou,\Pi_k^{\partial}u;\bm r_1,w_1,\mu_1)+\mathscr C(\bm p,\bm p; \Pi_{k+1}^ou,\Pi_k^{\partial}u; w_1)\\
	&=(Mu - u_t, w_1)_{\mathcal{T}_h}+ \langle (\bm{\Pi}^o_{k}\bm q-\bm q)\cdot \bm n,w_1 - \mu_1
	\rangle_{\partial\mathcal{T}_h} + ( \bm p (\Pi_{k+1}^ou - u), \nabla w_1)_{\mathcal{T}_h}\\
	&\quad- \langle  \bm p \cdot\bm n (\Pi_k^{\partial}u -u), w_1 - \mu_1 \rangle_{\partial\mathcal{T}_h} +\langle h_K^{-1}(\Pi_{k+1}^o u-u),\Pi_{k}^{\partial}w_1-\mu_1\rangle_{\partial\mathcal{T}_h}.
	\end{align*}
	holds for all $(\bm r_1, w_1, \mu_1)\in \bm Q_h\times V_h\times\widehat V_{h}(0)$.
\end{lemma}
\begin{proof} 
	By the definition of ${\mathscr {A}}$ and $\mathscr C$ {in (\ref{def_A}) and (\ref{def_C}) respectively, the projections  and integrating by parts}, we get
	\begin{align*}
	\hspace{1em}&\hspace{-1em} \mathscr A(\bm{\Pi}_k^o\bm q, \Pi_{k+1}^ou,\Pi_k^{\partial}u; \bm r_1,w_1,\mu_1)\\
	&=\langle (\bm{\Pi}^o_{k}\bm q-\bm q)\cdot \bm n,w_1 - \mu_1
	\rangle_{\partial\mathcal{T}_h} + (\nabla\cdot \bm q, w_1)_{\mathcal T_h}\\
	&\quad+\langle h_K^{-1}(\Pi_{k+1}^o u-u),\Pi_{k}^{\partial}w_1-\mu_1\rangle_{\partial\mathcal{T}_h},
	\end{align*}
	{where we have also used (\ref{Drift_Diffusion_Main_Equation_Mixed_Weak_Form_a})}. In addition,
	\begin{align*}
	\mathscr C(\bm p,\bm p; \Pi_{k+1}^ou,\Pi_k^{\partial}u; w_1)=( \bm p \Pi_{k+1}^ou, \nabla w_1)_{\mathcal{T}_h} - \langle  \bm p \cdot\bm n \Pi_k^{\partial}u, w_1 \rangle_{\partial\mathcal{T}_h}.
	\end{align*}
	Hence, { again using the projections}, we have 
	\begin{align*}
	\hspace{0.1em}&\hspace{-0.1em} M(\Pi_{k+1}^ou, w_1)_{\mathcal{T}_h}
	+\mathscr A(\bm{\Pi}_k^o\bm q, \Pi_{k+1}^ou,\Pi_k^{\partial}u;\bm r_1,w_1,\mu_1)+\mathscr C(\bm p; \Pi_{k+1}^ou,\Pi_k^{\partial}u; w_1)\\
	&=(Mu, w_1)_{\mathcal{T}_h}+ \langle (\bm{\Pi}^o_{k}\bm q-\bm q)\cdot \bm n,w_1 - \mu_1
	\rangle_{\partial\mathcal{T}_h} + (\nabla\cdot \bm q, w_1)_{\mathcal T_h}\\
	&\quad+\langle h_K^{-1}(\Pi_{k+1}^o u-u),\Pi_{k}^{\partial}w_1-\mu_1\rangle_{\partial\mathcal{T}_h} + ( \bm p \Pi_{k+1}^ou, \nabla w_1)_{\mathcal{T}_h} - \langle  \bm p \cdot\bm n \Pi_k^{\partial}u, w_1 \rangle_{\partial\mathcal{T}_h}.
	\end{align*}
	Since, {using (\ref{Drift_Diffusion_Main_Equation_Mixed_Weak_Form_c}),} $\nabla\cdot\bm q = \nabla\cdot(\bm p u) - u_t$, then we have 
	\begin{align*}
	(\nabla\cdot\bm q, w_1)_{\mathcal T_h} = -(u_t, w_1)_{\mathcal T_h} + \langle \bm p\cdot\bm n u, w_1  \rangle_{\partial \mathcal T_h} - (\bm pu, \nabla u)_{\mathcal T_h}.
	\end{align*}
	This implies that
	\begin{align*}
	\hspace{0.1em}&\hspace{-0.1em} M(\Pi_{k+1}^ou, w_1)_{\mathcal{T}_h}
	+\mathscr A(\bm{\Pi}_k^o\bm q, \Pi_{k+1}^ou,\Pi_k^{\partial}u;\bm r_1,w_1,\mu_1)+\mathscr C(\bm p; \Pi_{k+1}^ou,\Pi_k^{\partial}u; w_1)\\
	&=(Mu - u_t, w_1)_{\mathcal{T}_h}+ \langle (\bm{\Pi}^o_{k}\bm q-\bm q)\cdot \bm n,w_1 - \mu_1
	\rangle_{\partial\mathcal{T}_h} + ( \bm p (\Pi_{k+1}^ou - u), \nabla w_1)_{\mathcal{T}_h}\\
	&\quad- \langle  \bm p \cdot\bm n (\Pi_k^{\partial}u -u), w_1 - \mu_1 \rangle_{\partial\mathcal{T}_h} +\langle h_K^{-1}(\Pi_{k+1}^o u-u),\Pi_{k}^{\partial}w_1-\mu_1\rangle_{\partial\mathcal{T}_h}.
	\end{align*}
	{and completes the proof of the lemma.}
\end{proof}
To simplify  notation, we define
\begin{align*}
\eta_h^{\bm q}:=\bm{\Pi}_k^o\bm q-\bm q_{Ih},\ \ \
\eta_h^{u}:=\Pi_{k+1}^ou-u_{Ih},\ \ \
\eta_h^{\widehat u}:=\Pi_k^{\partial}u-\widehat u_{Ih}.
\end{align*}
We then subtract the equation in \Cref{projetion_error} from \eqref{projection01} to get the following lemma.
\begin{lemma} {Under the conditions of Lemma \ref{projetion_error}}, 
	we have the error equation
	\begin{align}\label{Appendix_error_equation}
	\begin{split}
	\hspace{1em}&\hspace{-1em} M(\eta_h^u,w_1)_{\mathcal{T}_h}+ \mathscr A (\eta_h^{\bm q}, \eta_h^{u},\eta_h^{\widehat u};\bm r_1,w_1, \mu_1)+\mathscr C(\bm p,\bm p; \eta_h^{u},\eta_h^{\widehat u};w_1)\\
	&= \langle (\bm{\Pi}^o_{k}\bm q-\bm q)\cdot \bm n,w_1 - \mu_1
	\rangle_{\partial\mathcal{T}_h} + ( \bm p (\Pi_{k+1}^ou - u), \nabla w_1)_{\mathcal{T}_h}\\
	&\quad- \langle  \bm p \cdot\bm n (\Pi_k^{\partial}u -u), w_1 \rangle_{\partial\mathcal{T}_h} +\langle h_K^{-1}(\Pi_{k+1}^o u-u),\Pi_{k}^{\partial}w_1-\mu_1\rangle_{\partial\mathcal{T}_h}.
	\end{split}
	\end{align}
	holds for all $(\bm r_1, w_1, \mu_1)\in \bm Q_h\times V_h\times\widehat V_{h}(0)$.
\end{lemma}

\subsection{Main error estimate}
{We can now prove (\ref{HDG_elliptic_projection_approximation2}).}
\begin{lemma} 
	For $h$ small enough, we have the error estimates
	\begin{align*}
	\|{\bm q}-{\bm q}_{Ih} \|_{\mathcal{T}_h}+\|h_K^{-1/2}(\Pi_k^{\partial}u_{Ih}-\widehat u_{Ih} )\|_{\partial\mathcal{T}_h}\le Ch^{k+1}|u|_{k+2}.
	\end{align*}
\end{lemma}
\begin{proof}
	We take $(\bm r_1, w_1, \mu_1)  = (\eta_h^{\bm q},\eta_h^u,\eta_h^{\widehat u})$ in \eqref{Appendix_error_equation}. First
	\begin{align*}
	\mathscr A (\eta_h^{\bm q},\eta_h^u,\eta_h^{\widehat u};\eta_h^{\bm q},\eta_h^u,\eta_h^{\widehat u}) = \|\eta_h^{\bm q}\|_{\mathcal T_h}^2 + \|h_K^{-1/2}(\Pi_k^\partial\eta_h^u-\eta_h^{\widehat u}) \|_{\partial\mathcal T_h}^2.
	\end{align*}
	Next, 
	\begin{align*}
	\hspace{1em}&\hspace{-1em} M(\eta_h^u,\eta_h^u)_{\mathcal{T}_h}+\mathscr C(\bm p,\bm p; \eta_h^{u},\eta_h^{\widehat u};\eta_h^{u})\\
	& =M(\eta_h^u,\eta_h^u)_{\mathcal{T}_h} +  ( \bm p\eta_h^u, \nabla \eta_h^u)_{\mathcal{T}_h} - \langle  \bm p \cdot\bm n \eta_h^{\widehat u}, \eta_h^u \rangle_{\partial\mathcal{T}_h}\\
	&=(M-\frac{1}{2}\nabla\cdot \bm p,\eta_h^u\eta_h^u)_{\mathcal{T}_h} + \frac 1 2 \langle\bm p\cdot\bm n \eta_h^u,\eta_h^u \rangle_{\partial\mathcal{T}_h}- \langle  \bm p \cdot\bm n \eta_h^{\widehat u}, \eta_h^u \rangle_{\partial\mathcal{T}_h}\\
	&=(M-\frac{1}{2}\nabla\cdot \bm p,\eta_h^u\eta_h^u)_{\mathcal{T}_h}+\frac{1}{2}\langle\bm p\cdot\bm n (\eta_h^u-\eta_h^{\widehat u}),\eta_h^u-\eta_h^{\widehat u} \rangle_{\partial\mathcal{T}_h}\\
	&\ge \frac{M}{2}\|\eta_h^u\|_{\mathcal{T}_h}
	-\frac{1}{2}\||\bm p\cdot\bm n|(\Pi_k^{\partial}\eta_h^u-\eta_h^{\widehat u})\|^2_{\partial\mathcal{T}_h}-Ch\|\bm p\|_{0,\infty}\|\nabla\xi_h^{u}\|^2_{\mathcal{T}_h}.
	\end{align*}
	For $h$ small enough, we obtain
	\begin{align*}
	\hspace{1em}&\hspace{-1em} M(\eta_h^u,\eta_h^u)_{\mathcal{T}_h}+\mathscr A (\eta_h^{\bm q},\eta_h^u,\eta_h^{\widehat u};\eta_h^{\bm q},\eta_h^u,\eta_h^{\widehat u}) +\mathscr C(\bm p,\bm p; \eta_h^{u},\eta_h^{\widehat u};\eta_h^{u})\\
	&\ge
	\frac{1}{2}\left(M\|\eta_h^u\|_{\mathcal{T}_h}^2
	+\|\eta_h^{{\bm q}}\|^2_{\mathcal{T}_h}+\|h_K^{-1/2}(\Pi_k^{\partial}\eta_h^{u}-\eta_h^{\widehat u}   )\|^2_{\partial\mathcal{T}_h}\right).
	\end{align*}
	
	On the other hand, 
	\begin{align*}
	\hspace{1em}&\hspace{-1em} M(\eta_h^u,\eta_h^u)_{\mathcal{T}_h}+\mathscr A (\eta_h^{\bm q},\eta_h^u,\eta_h^{\widehat u};\eta_h^{\bm q},\eta_h^u,\eta_h^{\widehat u}) +\mathscr C(\bm p,\bm p; \eta_h^{u},\eta_h^{\widehat u};\eta_h^{u})\\
	&=\langle (\bm{\Pi}^o_{k}\bm q-\bm q)\cdot \bm n,\eta_h^u-\eta_h^{\widehat u}
	\rangle_{\partial\mathcal{T}_h} + ( \bm p (\Pi_{k+1}^ou - u), \nabla \eta_h^u)_{\mathcal{T}_h}\\
	&\quad- \langle  \bm p \cdot\bm n (\Pi_k^{\partial}u -u), \eta_h^u -\eta_h^{\widehat u} \rangle_{\partial\mathcal{T}_h} +\langle h_K^{-1}(\Pi_{k+1}^o u-u),\Pi_{k}^{\partial}\eta_h^u-\eta_h^{\widehat u}\rangle_{\partial\mathcal{T}_h}\\
	&=:R_1+R_2+R_3+R_4.
	\end{align*}
	
	Next, we estimate $\{R_i\}_{i=1}^4$ term by term.  For the first term $R_1$,  \Cref{energy_argument1}  gives 
	\begin{align*}
	R_1&\le Ch^{k+1}|\bm q|_{k+1}\|h_K^{-1/2}(\eta_h^u-\eta_h^{\widehat u})\|_{\partial\mathcal{T}_h},\\%
	&\le  Ch^{k+1}|\bm q|_{k+1}\left(\|\eta_h^{\bm q}\|_{\mathcal{T}_h}+
	\|h_K^{-1/2}(\Pi_k^{\partial}\eta_h^u-\eta_h^{\widehat u})\|_{\partial\mathcal{T}_h}\right).
	\end{align*}
	For the term $R_2$, by  \Cref{HDG_Poincare} and \Cref{energy_argument1} to get 
	\begin{align*}
	R_2&\le Ch^{k+2}|u|_{k+2}\|\nabla\eta_h^u\|_{\mathcal{T}_h}\\
	&\le  Ch^{k+2}|u|_{k+2}\left(\|\eta_h^{\bm p}\|_{\mathcal{T}_h}+\|h_K^{-1/2}(\Pi_k^{\partial}\eta_h^u-\eta_h^{\widehat u})\|_{\partial\mathcal{T}_h}\right).
	\end{align*}
	For the term $R_3$, we use \Cref{energy_argument1}  to get 
	\begin{align*}
	R_3  &= \langle \bm p\cdot\bm n (\Pi_k^{\partial} u-u), \eta_h^u-\eta_h^{\widehat u} \rangle_{\partial\mathcal{T}_h}\\
	&\le  Ch^{k+1}|u|_{k+1}\|h_K^{-1/2}(\eta_h^u-\eta_h^{\widehat u})\|_{\partial\mathcal{T}_h}\\
	&\le  Ch^{k+1}|u|_{k+1}\left(
	\|\eta_h^{\bm p}\|_{\mathcal{T}_h}+
	\|h_K^{-1/2}(\Pi_k^{\partial}\eta_h^u-\eta_h^{\widehat u})\|_{\partial\mathcal{T}_h}
	\right).
	\end{align*}
	Moreover, for the last  term we have 
	\begin{align*}
	R_4 \le Ch^{k+1}|u|_{k+1}  \|h_K^{-1/2}(\Pi_k^{\partial}\eta_h^u-\eta_h^{\widehat u})\|_{\partial\mathcal{T}_h}.
	\end{align*}
	Use the Cauchy-Schwarz inequality for the above estimates of $\{R_i\}_{i=1}^4$, we get 
	\begin{align*}
	\|\eta_h^{{\bm q}}\|_{\mathcal{T}_h}+\|h_K^{-1/2}(\Pi_k^{\partial}\eta_h^{u}-\eta_h^{\widehat u}   )\|_{\partial\mathcal{T}_h}
	\le Ch^{k+1}|u|_{k+2}.
	\end{align*}
	{Use of the triangle inequality and estimates (\ref{eq27a}) and (\ref{eq27b}) completes the estimate.}
\end{proof}

\subsection{Duality arguments}
To obtain a $L^2$ norm estimate  of $\|\eta_h^{u}\|_{\mathcal T_h}$, we 
use the dual problem (\ref{D2}) with corresponding a priori estimate
(\ref{elliptic_regularity}).	
To perform the error analysis, the main difficulty is  to deal with the nonlinearity. We  define a new form $\mathscr C^\star$ which is related to the trilinear form $\mathscr C$:
\begin{align}\label{def_C_star}
\begin{split}
{\mathscr{C}}^{\star}(\bm p,\bm p; u_h,\widehat u_h;w_1)= - ( \bm p u_h, \nabla w_1)_{\mathcal{T}_h} +\langle \bm p \cdot\bm n \widehat u_h, w_1\rangle_{\partial\mathcal{T}_h}-(\nabla\cdot\bm pu_h,w_1)_{\mathcal{T}_h}.
\end{split}
\end{align}

Next, we give a property of the operators $\mathscr C$ and $\mathscr C^\star$. We omit the proof  since it is very straightforward.
\begin{lemma}\label{C_property1}
	For all $(u_h ,\widehat u_h, w_1,\mu_1)\in  V_h\times\widehat V_{h}(0)\times V_h\times\widehat V_{h}(0)$, we have 
	\begin{align*}
	\mathscr C(\bm p,\bm p;u_h,\widehat u_h;w_1)+ \mathscr C^{\star}(\bm p,\bm p;w_1,\mu_1; -u_h)
	=\langle \bm p\cdot\bm n(u_h-\widehat u_h), w_1-\mu_1 \rangle_{\partial\mathcal{T}_h}.
	\end{align*}
\end{lemma}

Similarly to \Cref{projetion_error},  we have the following lemma.
\begin{lemma}\label{projetion_error_dual} {Assuming $M$ is chosen sufficiently large, let $({\Phi},\bm \Psi)$ solve (\ref{D2}) then}
	we have the equation
	\begin{align*}
	\hspace{1em}&\hspace{-1em} M({\Pi}_{k+1}^o \Phi,w_1)_{\mathcal{T}_h}+
	\mathscr A(\bm{\Pi}_k^o\bm{\Psi},\Pi_{k+1}^o\Phi,\Pi_k^{\partial}\Phi;\bm r_1, w_1,\mu_1) + \mathscr C^\star(\bm p,\bm p; \Pi_{k+1}^o\Phi,\Pi_k^{\partial}\Phi;w_1)\\
	&= (\Theta,w_1)+\langle(\bm{\Pi}_k^o\bm \Psi-\bm \Psi\cdot)\bm n,w_1-\mu_1
	\rangle_{\partial\mathcal{T}_h} +  \langle h_K^{-1}(\Pi_{k+1}^o\Phi-\Phi),\Pi_k^{\partial}w_1-\mu_1
	\rangle_{\partial\mathcal{T}_h}\\
	&\quad-( \bm p (\Pi_{k+1}^o\Phi-\Phi), \nabla w_1)_{\mathcal{T}_h}
	+\langle \bm p\cdot\bm n (\Pi_k^{\partial} \Phi-\Phi), w_1 \rangle_{\partial\mathcal{T}_h} \\
	&\quad - (\nabla\cdot\bm p(\Pi_{k+1}^o\Phi-\Phi),w_1)_{\mathcal{T}_h}.
	\end{align*}
	holds for all $(\bm r_1, w_1, \mu_1)\in \bm Q_h\times V_h\times\widehat V_{h}(0)$.
\end{lemma}
{With the above preparation we can now derive estimate (\ref{HDG_elliptic_projection_approximation1}).}
\begin{theorem} Let $u$ and $u_{Ih}$ be the solutions of \eqref{Drift_Diffusion_Main_Equation_Mixed_Weak_Form} and \eqref{projection01}, respectively. If $h$ is small enough, then we have the error estimate
	\begin{align*}
	\|u-u_{Ih}\|_{\mathcal{T}_h}\le Ch^{k+2}\|u\|_{k+2}.
	\end{align*}
\end{theorem}
\begin{proof} 
	We take $(\bm r_1,w_1, \mu_1)=(\eta_h^{\bm q},-\eta_h^u,-\eta_h^{\widehat u})$ and $\Theta=-\eta_h^u$ in \Cref{projetion_error_dual} to get
	\begin{align*}
	\hspace{1em}&\hspace{-1em} -M({\Pi}_{k+1}^o\Phi,\eta_h^u)_{\mathcal{T}_h}+
	\mathscr A(\bm{\Pi}_k^o\bm{\Psi},\Pi_{k+1}^o\Phi,\Pi_k^{\partial}\Phi;\eta_h^{\bm q},-\eta_h^u,-\eta_h^{\widehat u}) + \mathscr C^\star(\bm p; \Pi_{k+1}^o\Phi,\Pi_k^{\partial}\Phi;-\eta_h^u)\\
	& = -M(\eta_h^u,{\Pi}_{k+1}^o\Phi)_{\mathcal{T}_h}-
	\mathscr A(\eta_h^{\bm q},\eta_h^u,\eta_h^{\widehat u}; -\bm{\Pi}_k^o\bm{\Psi},\Pi_{k+1}^o\Phi,\Pi_k^{\partial}\Phi) + \mathscr C^\star(\bm p; \Pi_{k+1}^o\Phi,\Pi_k^{\partial}\Phi;-\eta_h^u)\\
	& = -M(\eta_h^u,{\Pi}_{k+1}^o\Phi)_{\mathcal{T}_h}-\langle (\bm{\Pi}^o_{k}\bm q-\bm q)\cdot \bm n,\Pi_{k+1}^o\Phi - \Pi_k^{\partial}\Phi
	\rangle_{\partial\mathcal{T}_h}  -( \bm p (\Pi_{k+1}^ou - u), \nabla \Pi_{k+1}^o\Phi)_{\mathcal{T}_h}\\
	&\quad+ \langle  \bm p \cdot\bm n (\Pi_k^{\partial}u -u), \Pi_{k+1}^o\Phi- \Pi_k^{\partial}\Phi \rangle_{\partial\mathcal{T}_h} -\langle h_K^{-1}(\Pi_{k+1}^o u-u),\Pi_{k}^{\partial} \Pi_{k+1}^o\Phi- \Pi_k^{\partial}\Phi\rangle_{\partial\mathcal{T}_h}  \\
	&\quad +M(\eta_h^u,{\Pi}_{k+1}^o\Phi)_{\mathcal{T}_h}+\mathscr C(\bm p; \eta_h^{u},\eta_h^{\widehat u};{\Pi}_{k+1}^o\Phi) + \mathscr C^\star(\bm p; \Pi_{k+1}^o\Phi,\Pi_k^{\partial}\Phi;-\eta_h^u).	
	\end{align*}
	By \Cref{C_property1} we have 
	\begin{align*}
	\mathscr C(\bm p,\bm p; \eta_h^{u},\eta_h^{\widehat u};{\Pi}_{k+1}^o\Phi) + \mathscr C^\star(\bm p,\bm p; \Pi_{k+1}^o\Phi,\Pi_k^{\partial}\Phi;-\eta_h^u) = \langle \bm p\cdot\bm n (\eta_h^{u}-\eta_h^{\widehat u}), \Pi_{k+1}^o\Phi-\Pi_k^{\partial}\Phi\rangle_{\partial \mathcal T_h}.
	\end{align*}
	This implies
	\begin{align*}
	\hspace{1em}&\hspace{-1em} -M({\Pi}_{k+1}^o\Phi,\eta_h^u)_{\mathcal{T}_h}+
	\mathscr A(\bm{\Pi}_k^o\bm{\Psi},\Pi_{k+1}^o\Phi,\Pi_k^{\partial}\Phi;\eta_h^{\bm q},-\eta_h^u,-\eta_h^{\widehat u}) + \mathscr C^\star(\bm p,\bm p; \Pi_{k+1}^o\Phi,\Pi_k^{\partial}\Phi;-\eta_h^u)\\
	& =-\langle (\bm{\Pi}^o_{k}\bm q-\bm q)\cdot \bm n,\Pi_{k+1}^o\Phi - \Pi_k^{\partial}\Phi
	\rangle_{\partial\mathcal{T}_h}  -( \bm p (\Pi_{k+1}^ou - u), \nabla \Pi_{k+1}^o\Phi)_{\mathcal{T}_h}\\
	&\quad+ \langle  \bm p \cdot\bm n (\Pi_k^{\partial}u -u), \Pi_{k+1}^o\Phi- \Pi_k^{\partial}\Phi \rangle_{\partial\mathcal{T}_h} -\langle h_K^{-1}(\Pi_{k+1}^o u-u),\Pi_{k}^{\partial} \Pi_{k+1}^o\Phi- \Pi_k^{\partial}\Phi\rangle_{\partial\mathcal{T}_h}  \\
	&\quad  +\langle \bm p\cdot\bm n (\eta_h^u - \eta_h^{\widehat u} ), \Pi_{k+1}^o\Phi- \Pi_k^{\partial}\Phi \rangle_{\partial \mathcal T_h}.
	\end{align*}
	On the other hand, we have 
	\begin{align*}
	\hspace{1em}&\hspace{-1em} -M({\Pi}_{k+1}^o\Phi,\eta_h^u)_{\mathcal{T}_h}+
	\mathscr A(\bm{\Pi}_k^o\bm{\Psi},\Pi_{k+1}^o\Phi,\Pi_k^{\partial}\Phi;\eta_h^{\bm q},-\eta_h^u,-\eta_h^{\widehat u}) + \mathscr C^\star(\bm p,\bm p; \Pi_{k+1}^o\Phi,\Pi_k^{\partial}\Phi;-\eta_h^u)\\
	& = -\|\eta_h^u\|_{\mathcal T_h}^2- \langle(\bm{\Pi}_k^o\bm \Psi-\bm \Psi\cdot)\bm n, \eta_h^u,-\eta_h^{\widehat u}
	\rangle_{\partial\mathcal{T}_h} - \langle h_K^{-1}(\Pi_{k+1}^o\Phi-\Phi),\Pi_k^{\partial}\eta_h^u-\eta_h^{\widehat u}
	\rangle_{\partial\mathcal{T}_h}\\
	&\quad+( \bm p (\Pi_{k+1}^o\Phi-\Phi), \nabla \eta_h^u)_{\mathcal{T}_h}
	-\langle \bm p\cdot\bm n (\Pi_k^{\partial} \Phi-\Phi), \eta_h^u - \eta_h^{\widehat u} \rangle_{\partial\mathcal{T}_h} + (\nabla\cdot\bm p(\Pi_{k+1}^o\Phi-\Phi),\eta_h^u)_{\mathcal{T}_h}.
	\end{align*}
	{Comparing the above two equations, we} get 
	\begin{align*}
	\|\eta_h^u\|^2_{\mathcal{T}_h} &= -\langle \bm{\Pi}_k^o\bm q\cdot\bm n-\bm q\cdot\bm n,\Pi_{k+1}^o\Phi-\Pi_k^{\partial}\Phi \rangle_{\partial\mathcal{T}_h} -\langle
	h_K^{-1}(\Pi_{k+1}^ou-u),\Pi_k^{\partial}\Pi_{k+1}^o\Phi-\Pi_k^{\partial}\Phi
	\rangle_{\partial\mathcal{T}_h}\\
	&-(\bm p (\Pi_{k+1}^ou-u), \nabla \Pi_{k+1}^o\Phi)_{\mathcal{T}_h}
	+\langle \bm p\cdot\bm n (\Pi_k^{\partial} u-u), \Pi_{k+1}^o\Phi  - \Pi_{k}^\partial \Phi \rangle_{\partial\mathcal{T}_h}\\
	&+\langle \bm p\cdot\bm n (\eta_h^u-\eta_h^{\widehat u}), \Pi_{k+1}^o\Phi-\Pi_{k}^{\partial}\Phi \rangle_{\partial\mathcal{T}_h}\\
	&-\langle\bm{\Pi}_k^o\bm \Psi\cdot\bm n-\bm \Psi\cdot\bm n,\eta_h^{\widehat u}-\eta_h^u
	\rangle_{\partial\mathcal{T}_h}- \langle
	h_K^{-1}(\Pi_{k+1}^o\Phi-\Phi),\Pi_k^{\partial}\eta_h^u-\eta_h^{\widehat u}
	\rangle_{\partial\mathcal{T}_h}\\
	&+( \bm p (\Pi_{k+1}^o\Phi-\Phi), \nabla \eta_h^u)_{\mathcal{T}_h}
	-\langle \bm p\cdot\bm n (\Pi_k^{\partial} \Phi-\Phi), \eta_h^u -  \eta_h^{\widehat u} \rangle_{\partial\mathcal{T}_h}\\
	&+(\nabla\cdot\bm p (\Pi_{k+1}^o\Phi-\Phi),\eta_h^u)_{\mathcal{T}_h}\nonumber\\
	=&\sum_{i=1}^{10}S_i.
	\end{align*}
	We estimate $\{S_i\}_{i=1}^{10}$  {as follows (we omit some of the details):}
	\begin{align*}
	S_1 &=-	\langle \bm{\Pi}_k^o\bm q\cdot\bm n-\bm q\cdot\bm n,\Phi-\Pi_{k+1}^{o}\Phi
	\rangle_{\partial\mathcal{T}_h} \le Ch^{k+2}|\bm q|_{k+1}\|\Phi\|_{2},\\
	S_2&=	- \langle h_K^{-1}(\Pi_{k+1}^ou-u),\Pi_{k+1}^{o}\Phi-\Phi
	\rangle_{\partial\mathcal{T}_h}\le Ch^{k+2}  |u|_{k+2}\|\Phi\|_{2},\\
	S_3&=-(\bm p (\Pi_{k+1}^ou-u), \nabla \Pi_{k+1}^o\Phi)_{\mathcal{T}_h} \le Ch^{k+2}|u|_{k+2}|\Phi|_1,\\
	S_4&=\langle \bm p\cdot\bm n(\Pi_k^{\partial}u-u),\Pi_{k+1}^o\Phi-\Phi\rangle_{\partial\mathcal{T}_h}\le Ch^{k+2}|u|_{k+1}\|\Phi\|_{2},\\
	S_5&\le C\|h_K^{-1/2}(\eta_h^u-\eta_h^{\widehat u})\|_{\partial\mathcal{T}_h}h|\Phi|_1\le Ch^{k+2}|u|_{k+2}|\Phi|_1,\\
	S_6&\le Ch\|h_K^{-1/2}(\eta_h^u-\eta_h^{\widehat u})\|_{\partial\mathcal{T}_h}\|\bm{\Psi}\|_1
	\le Ch^{k+2 }\|\bm{\Psi}\|_{1},\\
	S_7&\le Ch\|h_K^{-1/2}(\eta_h^u-\eta_h^{\widehat u})\|_{\partial\mathcal{T}_h}\|\bm{\Phi}\|_{2}
	\le Ch^{k+2}|u|_{k+2}\|\bm{\Phi}\|_{2},\\
	S_8&\le Ch^{2}\|\Phi\|_{2}\|\nabla \eta_h^u\|_{\mathcal{T}_h}\le Ch^{k+2}\|\Phi\|_{2}|u|_{k+2},\\
	S_9&= - \langle \bm p\cdot\bm n(\Pi_k^{\partial}\Phi-\Phi),\eta_h^u-\eta_h^{\widehat u}\rangle_{\partial \mathcal T_h}\le Ch^{k+2}|\Phi|_1|u|_{k+2},\\
	S_{10}&\le Ch^{2}\|\Phi\|_{2}\|\eta_h^u\|_{\mathcal{T}_h}.
	\end{align*}
	{Summing the above estimates, we get }
	\begin{align*}
	\|\eta_h^u\|^2_{\mathcal{T}_h}\le Ch^{k+2}\|u\|_{k+2}	\|\eta_h^u\|_{\mathcal{T}_h}
	+Ch^2	\|\eta_h^u\|^2_{\mathcal{T}_h}
	\end{align*}
	Let $h$ be small enough, we have 
	\begin{align*}
	\|\eta_h^u\|_{\mathcal{T}_h}\le Ch^{k+2}\|u\|_{k+2}.	
	\end{align*}
	A simple application of the triangle inequality finishes the proof.
\end{proof}

\bibliographystyle{plain}
\bibliography{Model_Order_Reduction,Ensemble,HDG,Interpolatory,Mypapers,Added,Drift_Diffusion}

\end{document}